\documentclass{amsart}
\usepackage[utf8]{inputenc}
\usepackage{amsfonts}
\usepackage{amsmath}
\usepackage{amsthm}
\usepackage{amssymb}
\usepackage{latexsym}
\usepackage{enumerate}
\usepackage{centernot}
\usepackage{multirow}
\usepackage{float}
\usepackage{ragged2e}
\usepackage{diagbox}

\theoremstyle{plain}
\newtheorem{theorem}{Theorem}
\newtheorem{lemma}[theorem]{Lemma}
\newtheorem{proposition}[theorem]{Proposition}
\newtheorem{corollary}[theorem]{Corollary}

\newtheorem{conjecture}[theorem]{Conjecture}
\newtheorem{question}[theorem]{Question}
\newtheorem{problem}[theorem]{Problem}

\theoremstyle{definition}
\newtheorem{definition}[theorem]{Definition}

\numberwithin{equation}{section}

\newcommand{\forces}{\Vdash}
\newcommand{\embed}{{<\!\!\circ\;}}
\newcommand{\re}{{\upharpoonright}}

\newcommand{\A}{{\mathcal A}}
\newcommand{\B}{{\mathcal B}}
\newcommand{\M}{{\mathcal M}}
\newcommand{\N}{{\mathcal N}}
\renewcommand{\P}{{\mathcal P}}
\newcommand{\U}{{\mathcal U}}

\newcommand{\CC}{{\mathbb C}}
\newcommand{\PP}{{\mathbb P}}
\newcommand{\QQ}{{\mathbb Q}}
\newcommand{\RR}{{\mathbb R}}
\newcommand{\MM}{{\mathbb M}}

\newcommand{\bb}{{\mathfrak b}}
\newcommand{\cc}{{\mathfrak c}}
\newcommand{\dd}{{\mathfrak d}}
\newcommand{\jj}{{\mathfrak j}}
\newcommand{\ddd}{{\mathfrak{dd}}}
\newcommand{\rr}{{\mathfrak r}}
\newcommand{\rrr}{{\mathfrak{rr}}}
\newcommand{\sss}{{\mathfrak{s}}}
\newcommand{\pp}{{\mathfrak{p}}}

\newcommand{\cov}{{\mathsf{cov}}}
\newcommand{\non}{{\mathsf{non}}}
\newcommand{\cof}{{\mathsf{cof}}}
\newcommand{\all}{{\mathsf{all}}}
\newcommand{\osc}{{\mathsf{osc}}}
\newcommand{\rel}{{\mathsf{rel}}}

\newcommand{\Sym}{{\mathrm{Sym}}}

\newcommand{\sub}{\subseteq}
\newcommand{\sem}{\setminus}
\newcommand{\twoom}{2^\omega}

\newcommand{\omom}{\omega^\omega}
\newcommand{\omloms}{[\omega]^{<\omega}}
\newcommand{\omoms}{[\omega]^\omega}

\begin{document}

\title{Density cardinals}
\author{Christina Brech, J\"org Brendle, and M\'arcio Telles}
\thanks{The first author was partially supported by FAPESP grants (2016/25574-8 and 2023/12916-1).} 
\address{Departamento de Matemática, Instituto de Matemática e Estatística, Universidade de São Paulo,
Rua do Matão, 1010, 05508-090, São Paulo, Brazil}
\email{brech@ime.usp.br}
\thanks{The second author was partially supported by Grant-in-Aid for Scientific Research (C) 18K03398, Japan Society for the Promotion of Science.} 
\address{Graduate School of System Informatics, Kobe University,
Rokkodai 1-1, Nada, Kobe 657-8501, Japan}
\email{brendle@kobe-u.ac.jp}
\address{Departamento de Matemática, Faculdade de Formação de Professores, Universidade do Estado do Rio de Janeiro, Rua Doutor Francisco Portela, 1470, 24435-005, Rio de Janeiro, Brazil}
\email{marcio.telles@uerj.br}

\begin{abstract}
How many permutations are needed so that every infinite-coinfinite set of natural numbers with asymptotic density can be rearranged to no longer have the same density? We prove that the density number $\ddd$, which answers this question, is equal to the least size of a non-meager set of reals, $\non (\M)$. The same argument shows that a slight modification of the rearrangement number $\rrr$ of~\cite{BBBHHL20} is equal to $\non (\M)$, and similarly for a cardinal invariant related to large-scale topology introduced by Banakh~\cite{Ba23}, thus answering a question of the latter. We then consider variants of $\ddd$ given by restricting the possible densities of the original set and / or of the permuted set, providing lower and upper bounds for these cardinals and proving consistency of strict inequalities. We finally look at cardinals defined in terms of relative density and of asymptotic mean, and relate them to the rearrangement numbers of~\cite{BBBHHL20}.
\end{abstract}

\maketitle



\section{Introduction}


Asymptotic density of sets of natural numbers plays an important role
in many areas of mathematics; it is also closely connected to
Lebesgue measure on the real numbers. For quite some time, density
has been investigated from the point of view of set theory: the
density ideal ${\mathcal Z}$ is one of the classical analytic P-ideals,
which have been characterized by Solecki~\cite{So99} using lower
semicontinuous submeasures, and its cardinal invariants as well
as those of the quotient algebra $\P (\omega) / {\mathcal Z}$
have been extensively studied (see e.g.~\cite{HH07} or~\cite{Ra20}).
More recently, variants of the splitting number and the reaping
number defined in terms of density have been introduced and compared
with other cardinal invariants of the continuum (\cite{BHKLS23},
see also~\cite{FKL} and~\cite{Va}). Here we look at density from
another point of view: our starting point is the simple observation,
inspired by Riemann's Rearrangement Theorem, that any infinite-coinfinite
set of natural numbers can be rearranged into a set without density
or with arbitrary density in the closed interval $[0,1]$.
The question of how many permutations are necessary to be able 
to rearrange every infinite-coinfinite set naturally leads
to cardinal invariants (which we will call {\em density numbers})
analogous to the rearrangement numbers studied in~\cite{BBBHHL20}
(see also~\cite{BBH19} for closely related work). We start with 
basic definitions.

For $A \sub \omega$ let 
\[ d_n (A) = { \frac{| A \cap n |}{n} } \]
and define the {\em lower density} 
\[  \underline{d} (A) = \liminf_{n \to \infty} d_n (A) \]
and the {\em upper density} 
\[ \overline{d} (A)= \limsup_{n \to \infty} d_n (A) \]
of $A$. If $\underline{d} (A) = \overline{d} (A)$ we call the common value $d(A)$ the {\em (asymptotic) density} of $A$.
Otherwise $A$ does not have (asymptotic) density. Notice that finite sets have asymptotic density equal to $0$ and cofinite sets have asymptotic density equal to $1$.
On the other hand, the density of infinite-coinfinite sets can be changed by permutations of $\omega$ as shown by the following simple result:
\begin{proposition}   \label{density-permutation}
    Given an infinite and coinfinite set $A$, we have:
    \begin{enumerate}
        \item For every $r \in [0,1]$, there is $\pi \in \Sym (\omega)$ such that $d(\pi[A]) = r$;
        \item There is $\pi \in \Sym (\omega)$ such that $\pi[A]$ doesn't admit asymptotic density.
    \end{enumerate}
\end{proposition}
This is reminiscent of Riemann's Rearrangement Theorem:
\begin{theorem}[Riemann]
Given a conditionally convergent (c.c. for short) series $\sum_n a_n$, we have:
   \begin{enumerate}
       \item For every $r \in \RR \cup \{ - \infty, + \infty \}$, there is $\pi \in \Sym (\omega)$ such that $ \sum_n a_{\pi (n)} = r$;
       \item There is $\pi \in \Sym (\omega)$ such that $\sum_n a_{\pi(n)} $ diverges by oscillation.
   \end{enumerate}
\end{theorem}
The latter means that $- \infty \leq \liminf_k \sum_{n=0}^k a_n < \limsup_k \sum_{n=0}^k a_n \leq + \infty$. Riemann's Theorem motivates the
following definition from~\cite{BBBHHL20}.
\begin{definition} 
The {\em rearrangement number} $\rrr$ is the smallest cardinality of a family $\Pi \sub \Sym (\omega)$ such that for every c.c. series $\sum_n a_n$ there is 
$\pi \in \Pi$ such that $\sum_n a_{\pi (n)} \neq \sum_n a_n$ (where we allow the possibility that $\sum_n a_{\pi (n)}$ converges to 
$\pm \infty$ or diverges by oscillation).
\end{definition}
By Proposition~\ref{density-permutation} we may define analogously:
\begin{definition} 
The {\em (asymptotic) density number} $\ddd$ is the smallest cardinality of a family $\Pi \sub \Sym (\omega)$ such that for every infinite-coinfinite set $A \sub \omega$ 
with asymptotic density there is 
$\pi \in \Pi$ such that $d(\pi [A]) \neq d (A)$, i.e., either $\pi [A]$ doesn't admit asymptotic density or $\pi[A]$ has asymptotic density distinct from $d (A)$.
\end{definition}
One may wonder why we require $A$ to have asymptotic density though this is irrelevant in the Proposition. This is because we want to make the definition analogous to $\rrr$.
Alternatively, Riemann's Theorem can be reformulated as follows. Call a series $\sum_n a_n$ {\em potentially conditionally convergent} ({\em p.c.c.} for short)~\cite[Definition 30]{BBBHHL20}
if there is a permutation $\pi \in \Sym (\omega)$ such that $\sum_n a_{\pi (n)}$ is c.c. Notice that $\sum_n a_n$ is p.c.c. iff $a_n \to 0$, $\sum \{ a_n : a_n \geq 0 \}
= + \infty$, and $\sum \{ a_n : a_n \leq 0 \} = - \infty$. The rephrased version of Riemann's Theorem then says that given any p.c.c. $\sum_n a_n$ and any 
$r \in \RR \cup \{ - \infty, + \infty \}$ there is $\pi \in \Sym (\omega)$ such that $\sum_n a_{\pi (n)} = r$ and there is also $\pi \in \Sym (\omega)$ such that
$\sum_n a_{\pi (n)}$ diverges by oscillation.

Let us first see that $\ddd$ is indeed an uncountable cardinal.

\begin{lemma}
Let $(a_n : n \in \omega)$ be a sequence in $\omega$ with the property that for some $n_0 \in \omega$ and every $n \geq n_0$, $a_{n+1} -a_n > 2^n$. 
Then $A=\{a_n: n\in \omega \}$ is an infinite and coinfinite subset of $\omega$ such that $d(A)=0$. 
\end{lemma}
\begin{proof}
The property of the sequence immediately guarantees that $A$ is infinite and coinfinite. As for the asymptotic density, 
first notice that a simple inductive argument shows that for $n \geq n_0$, $a_{n+1} > a_{n_0} + 2^n$. 
Hence, for $a_{n_0} + 2^{n-1} \leq k \leq a_{n_0} + 2^n$, we have that $A \cap k \subseteq \{a_i:  i \leq n\}$, so 
$$\frac{|A \cap k  |}{k} \leq \frac{n + 1}{k} \leq \frac{n + 1}{a_{n_0}+2^{n-1}},$$
which converges to $0$, when $k$ goes to infinity.
\end{proof}

\begin{proposition}
    $\aleph_1 \leq \mathfrak{dd}$. 
\end{proposition}
\begin{proof}
Let $\{\pi_i: i \in \omega \}$ be a countable subset of $\Sym (\omega) $ and, without loss of generality, assume that $\pi_0$ is the identity map. 
We construct, inductively, a sequence of integers $(a_n :  n \in \omega)$ such that for every $i \leq n$, $\pi_i(a_{n+1}) - \pi_i(a_n) \geq 2^n$. 

Let $a_0 =1$ and given $a_n$, let $a_{n+1}>a_n$ be such that $\pi_i(a_{n+1}) \geq \pi_i(a_n)+ 2^n$ for all $i \leq n$. 
This can be done since $\lim_{k \rightarrow \infty} \pi_i(k) = \infty$ for every $i \in \omega$.

Let $A=\{a_n: n \in \omega \}$; then it follows from the previous lemma that $A$ is an infinite and coinfinite subset of $\omega$ such that $d(\pi_i[A])= d(A)=0$ for every $i \in \omega$.
\end{proof}

In the next section, we shall obtain a much better result (Theorem~\ref{density-nonmeager}): Tukey connections are used to establish that $\ddd$ equals $\non (\M)$, the smallest cardinality  of a nonmeager set of reals. We use similar methods to prove that a variant of $\rrr$ and a cardinal invariant related to large-scale topology are also equal to $\non (\M)$. In Section 3 we introduce variants of $\ddd$ by imposing restrictions on the original density of a set and/or the density of its image after a permutation. Interesting cases occur when these densities are imposed to be in one of the following sets: $\{0,1\}$, $(0,1)$, $[0,1]$ or $\{\mathrm{osc}\}$, where the last case means that the corresponding density does not exist. We then give ZFC lower and upper bounds for these cardinals in terms of  well-known cardinals: the already mentioned $\mathrm{non}(\mathcal{M})$; the bounding number $\mathfrak{b}$; the reaping number $\mathfrak{r}$; and the covering numbers of the meager and the null ideals, $\mathrm{cov}(\mathcal{M})$ and $\mathrm{cov}(\mathcal{N})$. This allows us to position our cardinals in Cicho\'n's diagram. Next, Section 4 provides some consistency results regarding strict inequalities. Finally, in Section 5, we consider variations of the rearrangement and the density cardinals defined within the framework of Section 3 and try to compare cardinals with analogous definitions. The similarity to rearrangement naturally leads us to consider relative density and the asymptotic mean. We also discuss the limitations of these analogies and pose several open questions.


Our notation is standard. For cardinal invariants not defined here we refer to~\cite{Bl10}. Only Section 4 requires knowledge of the forcing method. For forcing, and its connection to cardinal invariants, see~\cite{BJ95} and~\cite{Ha17}.



\section{The density number and the uniformity of the meager ideal}


Since the definitions of $\rrr$ and $\ddd$ are quite similar we expect that there is a close relationship between the cardinals. We shall see in this section
(Corollary~\ref{density-nonmeager-cor}) that $\rrr \leq \ddd$ while it remains open whether equality holds.

It is convenient to formulate order relations between cardinals in terms of {\em Tukey connections} because this will streamline the arguments and will
automatically give results about dual cardinals. Let us quickly review this language (\cite{V92}, \cite[Section 4]{Bl10}). 
Consider {\em relational systems} $\A = ( A_-, A_+, A)$ such that $A \sub A_- \times A_+$
is a relation, for every $a_- \in A_-$ there is $a_+ \in A_+$ with $a_- A a_+$, and for every $a_+ \in A_+$ there is $a_- \in A_-$ such that $a_- A a_+$ fails.
For such $\A$ there is the {\em dual} relational system $\A^\perp = (A_+ , A_- , \neg \breve A)$ where $\breve A$ is the converse of $A$, that is, $a_+ \neg\breve A a_-$ if
$\neg (a_- A a_+)$. With each relational system $\A = (A_-,A_+,A)$ we associate two cardinals: the {\em dominating number} $\dd (A_-,A_+,A)$ is the smallest
size of a family $C_+ \sub A_+$ such that for all $a_- \in A_-$ there is $a_+ \in C_+$ with $a_- A a_+$; the {\em unbounding number} $\bb (A_-, A_+ ,A)$
is the least cardinality of a family $C_- \sub A_-$ such that for all $a_+ \in A_+$ there is $a_- \in C_-$ such that $a_- A a_+$ fails. Notice that
$\dd (\A^\perp) = \bb (\A)$ and $\bb (\A^\perp) = \dd (\A)$. 

Given two relational systems $\A = (A_-, A_+, A)$ and $\B = (B_- , B_+ , B)$, $\A$ is {\em Tukey reducible} to $\B$ ($\A \leq_T B$ in symbols) if there exist
functions $\varphi_- : A_- \to B_-$ and $\varphi_+ : B_+ \to A_+$ such that $\varphi_- (a_-) B b_+$ implies $a_- A \varphi_+ (b_+)$ for all
$a_- \in A_-$ and $b_+ \in B_+$. $\A$ and $\B$ are {\em Tukey equivalent} ($\A \equiv_T B$ in symbols) if $\A \leq_T \B$ and $\B \leq_T \A$ both hold.
$\A \leq_T \B$ is equivalent to $\B^\perp \leq_T \A^\perp$. Either implies that $\dd (\A) \leq \dd(\B)$ and $\bb (\A) \geq \bb (\B)$.
Conversely, if $\dd (\A) > \dd(\B)$ is consistent then $\A \leq_T \B$ consistently fails. Note, however, that all Tukey connections we shall exhibit are Borel
and thus absolute. Hence, for Borel Tukey reducibilities, the consistency of $\dd (\A) > \dd (\B)$ means that $\A \leq_T \B$ fails in ZFC.

The relational system relevant for the density number is $(D , \Sym (\omega), R)$ where $D$ is the collection of infinite-coinfinite sets with density
and $R$ is given by $x R \pi$ if $d (\pi [x])  \neq d(x)$ for $x \in D$ and $\pi \in \Sym (\omega)$ (where we allow the possibility that $d (\pi [x])$ is undefined).
Clearly $\ddd = \dd (D,\Sym(\omega),R)$. Denote by $\ddd^\perp = \bb (D,\Sym (\omega), R)$ the dual cardinal, that is, the least size of a family $D_0$
of infinite-coinfinite sets with density such that for all $\pi \in \Sym (\omega)$ there is $x \in D_0$ such that $d (\pi [x]) = d(x)$. 

Let $\M$ and $\N$ be the meager and null ideals on the Cantor space $\twoom$, respectively. Consider the relational system $(\M, \twoom, \not\ni)$.
Then $\dd (\M, \twoom, \not\ni) = \non (\M)$ is the uniformity of the meager ideal and $\bb (\M, \twoom, \not\ni) = \cov (\M)$ is the covering number of the meager ideal.

We will also need the following relational system. Let $h \in \omom$ be a function with $|h (n)| \geq 1$ for all $n$. A function $\phi : \omega \to \omloms$
is called an {\em $h$-slalom} if $|\phi (n)| \leq h(n)$ for all $n$. Let $\Phi_h$ be the collection of all $h$-slaloms. Consider the triple
$(\omom, \Phi_h, \in^\infty)$ where $f \in^\infty \phi$ if $f(n) \in \phi(n)$ holds for infinitely many $n$, for $f \in \omom$ and $\phi \in \Phi_h$.
The connection between the invariants of $(\omom, \Phi_h, \in^\infty)$ and the invariants of $\M$ is established by the following classical
result of Bartoszy\'nski:
\begin{theorem}[Bartoszy\'nski {\cite{Ba87}, \cite[Lemmas 2.4.2 and 2.4.8]{BJ95}}]   \label{bartoszynski-characterization}
For any $h \in\omom$ satisfying $|h (n) | \geq 1$ for all $n \in \omega$ we have that
$\non (\M) ={ \dd (\omom, \Phi_h, \in^\infty)}$ and $\cov (\M) = \bb (\omom, \Phi_h, \in^\infty)$.
\end{theorem}
However, and this is the tricky part, $(\omom,\Phi_h,\in^\infty)$ and $(\M,\twoom,\not\ni)$ are not Tukey equivalent. While $(\omom,\Phi_h,\in^\infty) \leq_T
(\M,\twoom, \not\ni)$ is easy to see (see, again, \cite{Ba87} or \cite[Lemmas 2.4.2 and 2.4.8]{BJ95}), the converse fails. (This follows from Zapletal's result~\cite{Za14}
saying that there is a forcing adding a half-Cohen real (i.e., an infinitely often equal real) without adding a Cohen real.) But $(\M,\twoom,\not\ni)$ is Tukey
reducible to a sequential composition of $(\omom,\Phi_h,\in^\infty)$ (\cite{Ba87} or \cite[Lemmas 2.4.2 and 2.4.8]{BJ95}, see also~\cite[Theorem 5.9]{Bl10} for the special
case $h$ is the constant function with value 1).

In~\cite[Theorems 7, 8, and 11]{BBBHHL20}, $\max\{\cov (\N), \bb\} \leq \rrr \leq \non (\M)$ was proved, and the same inequalities hold for $\ddd$ instead of $\rrr$.
There is, however, a better result:

\begin{theorem} $(\omom,\Phi_h,\in^\infty) \leq_T (D, \Sym (\omega), R) \leq_T (\M, \twoom, \not\ni)$ for any $h \in\omom$ growing fast enough. \label{density-nonmeager}
\end{theorem}

\begin{corollary} $\ddd = \non (\M)$ and $\ddd^\perp = \cov (\M)$. A fortiori, $\rrr \leq \ddd$ holds. \label{density-nonmeager-cor}
\end{corollary}

\begin{proof} This is immediate by Theorems~\ref{density-nonmeager} and~\ref{bartoszynski-characterization}.
\end{proof}

\begin{proof}[Proof of Theorem~\ref{density-nonmeager}]
We first show the easier $(D, \Sym (\omega), R) \leq_T (\M, \twoom, \not\ni)$. Notice that for the meager ideal it does not matter whether we consider $\twoom$
or the Baire space $\omom$; furthermore $\Sym(\omega)$ is homeomorphic to  $\omom$; so we may as well work with the meager ideal on $\Sym (\omega)$
and with the triple $(\M, \Sym(\omega),\not\ni)$ instead of $(\M,\twoom,\not\ni)$. We need to define $\varphi_-: D \to \M$ and $\varphi_+ : \Sym (\omega) \to \Sym (\omega)$
such that for all $x \in D$ and $\pi \in \Sym (\omega)$, if $\pi \notin \varphi_- (x)$ then $x R \varphi_+ (\pi)$.

Let $\varphi_+$ be the identity function, $\varphi_+ (\pi) = \pi$ for $\pi \in \Sym (\omega)$, and, for $x \in D$, let $\varphi_-(x)=\bigcup_{k,\ell\in\omega}F_{k,\ell}$, where

\[ F_{k,\ell} = \left\{ \pi \in \Sym (\omega) : \left(\forall n\geq\ell \; {| \pi (x) \cap n| \over n} \geq {1 \over k} \mbox{ or } 
   \forall n\geq\ell \; {| \pi (x) \cap n| \over n} \leq 1 - {1 \over k} \right) \right\}. \]
To see that $\varphi_- (x)$ is meager in $\Sym (\omega)$, it suffices to show that each
$A_{k,\ell}=\Sym(\omega) \setminus F_{k,\ell}$ is open and dense. But, as

\[ A_{k,\ell} = \left\{ \pi \in \Sym (\omega) : \left(\exists n\geq\ell \; {| \pi (x) \cap n| \over n} \leq {1 \over k} \mbox{ and } 
   \exists n\geq\ell \; {| \pi (x) \cap n| \over n} \geq 1 - {1 \over k} \right) \right\}, \]
if $\pi\in A_{k,\ell}$, it is clear that any $\pi^\prime \in \Sym (\omega)$ that agrees with 
$\pi$ in a sufficiently large initial segment of $\omega$ also belongs to $A_{k,\ell}$.
Bearing in mind the topology of $\Sym (\omega )$, one sees that $A_{k,\ell}$ is open.
The fact that $A_{k,\ell}$ is dense amounts to showing that, for every $\pi \in \Sym (\omega)$
and $m\in\omega$, there is $\pi^\prime \in A_{k,\ell}$ such that $\pi \restriction_m = \pi^\prime \restriction_m$, but that is also clear.

Finally, if $\pi \notin \varphi_- (x)$, then $\underline{d} (\pi [x]) =0$ and $\overline{d} ( \pi [x]) = 1$ so that the density of $\pi[x]$ is not defined.
A fortiori, $x R \varphi_+ (\pi)$ holds.

Next let $h \in \omom$ be such that $h(n) \geq 2^n + n + 1$ for all $n$. To prove that $(\omom,\Phi_h,\in^\infty) \leq_T (D, \Sym (\omega), R) $ we need to define
$\varphi_- : \omom \to D$ and $\varphi_+ : \Sym (\omega) \to \Phi_h$ such that for all $g \in \omom$ and all $\pi \in \Sym (\omega)$, if
$\varphi_- (g) R \pi$ then $g \in^\infty \varphi_+ (\pi)$.

Let $\varphi_+(\pi)(n) = 2^n \cup \{ \pi^{-1} (k) : k \leq n \}$. Clearly $|\varphi_+ (\pi) (n)| \leq 2^n + n + 1 \leq h(n)$. If $g \in \omom$ satisfies $g(n) \geq 2^n$ for almost
all $n$, recursively define $\varphi_- (g) = \{ a_n^g : n \in \omega \}$ as follows: let $n_0$ be minimal such that $g(n) \geq 2^n$ for all $n \geq n_0$; then let
$a_0^g = 0$ and
\[ a_{n+1}^g = \left\{ \begin{array}{ll} 2^{a_n^g} & \mbox{ if } n < n_0 \\
   g( a_n^g ) & \mbox{ if } n \geq n_0 \end{array} \right. \]
Note that $a_{n+1}^g \geq 2^{a_n^g} \geq 2^n$ holds for all $n$, that the $a_n^g$ therefore form a strictly increasing sequence, and that $d(\varphi_- (g)) = 0$ holds.
(If $g$ is not of this form, the definition of $\varphi_- (g)$ is irrelevant.)

Assume that $g \in\omom$ and $\pi \in \Sym (\omega)$ are such that $g \in^\infty \varphi_+ (\pi)$ fails. Then $g(n) \notin \varphi_+ (\pi)(n)$ holds for almost all $n$.
In particular $g(n) \geq 2^n$ for almost all $n$ so that the definition of $\varphi_- (g)$ in the preceding paragraph applies. Let $n_0$ be minimal such that
$g(n) \notin \varphi_+ (\pi)(n)$ for all $n \geq n_0$. Then $a_{n+1}^g = g (a_n^g) \notin \varphi_+ (\pi) (a_n^g)$ holds for all $n \geq n_0$. Hence $a_{n+1}^g \neq \pi^{-1} (k)$
for $k \leq a_n^g$, and $\pi (a_{n+1}^g) > a_n^g$ follows. In particular $\pi (a^g_{n+2}) > a_{n+1}^g \geq 2^n$ for $n \geq n_0$, and $d( \pi [ \varphi_- (g)]) = 0$ is immediate.
Thus $\varphi_- (g) $ and $\pi [ \varphi_- (g)]$ have the same density, and $\varphi_- (g) R \pi$ fails, as required. This completes the proof of the theorem.
\end{proof}

Combinatorial characterizations of $\non (\M)$ like $\dd (\omom, \Phi_h, \in^\infty)$ have been used before to show that cardinals defined in terms of permutations
of $\omega$ are larger than or equal to $\non (\M)$, see~\cite[Theorems 2.2 and 2.4]{BSZ00} of which the present proof is reminiscent.

In view of the analogy between permutations of density and rearrangement of series, it may now be natural to conjecture:
\begin{problem}[{\cite[Question 46]{BBBHHL20}}]
Does $\rrr = \non (\M)$ hold?
\end{problem}
However, it is far from clear that this is true, for while sets of density strictly between $0$ and $1$ seem to correspond to c.c. series, sets of density $0$ or $1$ rather correspond to
p.c.c. series converging to $\pm \infty$: 

\begin{table}[htbp]
    \centering
\begin{tabular}{ |c|c| } 
\hline 
p.c.c. series $\sum_n a_n$ & infinite-coinfinite set $A$ \\ [0.5ex] 
\hline 
c.c. series & $d(A) \in (0,1)$ \\ [0.5ex] 
$\sum_n a_n = \pm \infty$ & $d(A) \in \{ 0,1 \}$ \\ [0.5ex] 
$\sum_n a_n$ diverges by oscillation & $d(A)$ is undefined \\ [0.5ex] 
\hline
\end{tabular} 
\vspace{5pt}
\caption{Analogy between rearrangement and density} \label{table1} 
\end{table}

Thus the following rearrangement analogue of the density number $\ddd$ is natural:

\begin{definition}

$\rrr'$ is the smallest cardinality of a family $\Pi \sub \Sym (\omega)$ such that for every p.c.c. series $\sum_n a_n \in \RR \cup \{ \pm \infty \}$ there is
$\pi \in \Pi$ such that $\sum_n a_{\pi (n)} \neq \sum_n a_n$, where, as usual, we allow the possibility that $\sum_n a_{\pi (n)}$ diverges by oscillation. 
\end{definition}

Clearly $\rrr \leq \rrr' \leq \non(\M)$ (the proof of the second inequality is
exactly like in~\cite[Theorem 8]{BBBHHL20}; this was originally proved by Agnew~\cite{Ag40}). The point is that:
\begin{theorem}
$\rrr' = \non(\M)$.
\end{theorem}
\begin{proof}[Proof Sketch]
The proof of $\non (\M) \leq \rrr'$ is analogous to the second part of the proof of Theorem~\ref{density-nonmeager}, and we therefore confine ourselves to
pointing out the necessary changes. We show $\dd (\omom, \Phi_h , \in^\infty) \leq \rrr'$ where $h$ is as in the latter proof. Given $\Pi \sub \Sym (\omega)$
with $|\Pi| < \dd (\omom, \Phi_h, \in^\infty)$, find $g \in\omom$ such that $g(n) \notin \varphi_+ (\pi)(n)$ for almost all $n$, for all $\pi \in \Pi$. Let $A: = \varphi_- (g) =
\{ a_n^g : n \in \omega \}$. Let $\{ b_n^g : n \in \omega \}$ be the increasing enumeration of the complement $B : = \omega \sem A$. Let
\[ c_k = \left\{ \begin{array}{ll} {1 \over n} & \mbox{ if } k = b_n^g \\
   - {1 \over n} & \mbox{ if } k = a_n^g \end{array} \right. \]
Clearly $\sum_k c_k$ is p.c.c. with $\sum_k c_k = \infty$. The argument of the proof of Theorem~\ref{density-nonmeager} now shows that $\sum_k c_{\pi(k)} = \infty$
still holds for all $\pi \in \Pi$. Thus $\Pi$ is not a witness for $\rrr '$, and $\non (\M) \leq \rrr'$ follows.
\end{proof}

For simplicity, we gave a direct, Tukey-free, proof of this result. For a discussion of the rearrangement number and its relatives in the framework of
Tukey connections, see van der Vlugt's master thesis~\cite{vdV19}.

We finally show that a similar method can be used to prove that a cardinal invariant related to large-scale topology is equal to $\non (\M)$.
Consider the triple $(\omoms, \Sym (\omega), S)$ where $x S \pi$ if the set $\{ n \in x : n \neq \pi (n) \in x \}$ is infinite. Following Banakh~\cite{Ba23} let
$\Delta = \dd ( \omoms, \Sym (\omega), S)$ and $\hat \Delta = \bb ( \omoms , \Sym (\omega) , S)$. That is, $\Delta$ is the least cardinality of
a $\Pi \sub \Sym (\omega)$ such that for all $x \in \omoms$ there is $\pi \in \Pi$ such that the set $\{ n \in x : n \neq \pi (n) \in x \}$ is infinite
while $\hat \Delta$ is the least cardinality of an $X \sub \omoms$ such that for all $\pi \in \Sym (\omega)$ there is $x \in X$ such that
the set $\{ n \in x : n \neq \pi (n) \in x \}$ is finite. Banakh~\cite[see Theorems 3.2 and 7.1]{Ba23} proved $(\omoms, \Sym (\omega), S) \leq_T (\M , \twoom , \not\ni)$ and established a connection
of $\Delta $ to large-scale topology. We prove:

\begin{theorem} \label{delta-nonmeager}
$(\omom, \Phi_h , \in^\infty) \leq_T ( \omoms , \Sym (\omega), S)$ for any $h \in \omom$ growing fast enough.
\end{theorem}

\begin{corollary}  \label{delta-nonmeager-cor}
$\Delta = \non (\M)$ and $\hat \Delta = \cov (\M)$.
\end{corollary}

This answers a question of Banakh~\cite[Problem 3.8]{Ba23}.

\begin{proof}[Proof of Corollary~\ref{delta-nonmeager-cor}]
This is immediate by Theorems~\ref{delta-nonmeager} and~\ref{bartoszynski-characterization}, and by Banakh's Theorems 3.2 and 7.1~\cite{Ba23}. 
\end{proof}

\begin{proof}[Proof of Theorem~\ref{delta-nonmeager}]
The proof is very similar to   the second part of the proof of Theorem~\ref{density-nonmeager}.
Let $h \in \omom$ be such that $h(n) \geq 3 (n+1)$ for all $n$.
We need to define $\varphi_- : \omom \to \omoms$ and $\varphi_+ : \Sym (\omega) \to \Phi_h$ such that for all $g \in \omom$ and $\pi \in \Sym (\omega)$, 
if $\varphi_- (g) S \pi$ then $g \in^\infty \varphi_+ (\pi)$.

Let $\varphi_+ (\pi)(n) = (n+1)  \cup \{ \pi^{-1} (k) , \pi (k) : k \leq n \}$. Clearly $|\varphi_+ (\pi) (n)| \leq 3 (n+1)  \leq h(n)$. If $g \in \omom$ satisfies $g(n) \geq n+1$ for almost
all $n$, recursively define $\varphi_- (g) = \{ a_n^g : n \in \omega \}$ as follows: let $n_0$ be minimal such that $g(n) \geq n+1$ for all $n \geq n_0$; then let
$a_0^g = 0$ and
\[ a_{n+1}^g = \left\{ \begin{array}{ll} {a_n^g} + 1 & \mbox{ if } n < n_0 \\
   g( a_n^g ) & \mbox{ if } n \geq n_0 \end{array} \right. \]
Note that $a_{n+1}^g \geq {a_n^g} + 1 \geq n + 1$ holds for all $n$, that the $a_n^g$ thus form a strictly increasing sequence, and that $\varphi_- (g)$ is infinite.
(If $g$ is not of this form, the definition of $\varphi_- (g)$ is irrelevant.)

Assume that $g \in\omom$ and $\pi \in \Sym (\omega)$ are such that $g \in^\infty \varphi_+ (\pi)$ fails. Then $g(n) \notin \varphi_+ (\pi)(n)$ holds for almost all $n$.
In particular $g(n) \geq n+1$ for almost all $n$ so that the definition of $\varphi_- (g)$ in the preceding paragraph applies. Let $n_0$ be minimal such that
$g(n) \notin \varphi_+ (\pi)(n)$ for all $n \geq n_0$. Fix such $n$ with $\pi (a_{n+1}^g) \neq a_{n+1}^g$. 
Then $a_{n+1}^g = g (a_n^g) \notin \varphi_+ (\pi) (a_n^g)$ holds. Hence $a_{n+1}^g \neq \pi^{-1} (k)$
for $k \leq a_n^g$, and $\pi (a_{n+1}^g) > a_n^g$ follows. Now assume $m > n+1$. Then $a_{m}^g = g (a_{m-1}^g) \notin \varphi_+ (\pi) (a_{m-1}^g)$ holds. 
Hence $a_{m}^g \neq \pi (k)$ for $k \leq a_{m-1}^g$, and $\pi (a_{n+1}^g) \neq a_m^g$ follows. Therefore $\pi (a_{n+1}^g)  \notin \varphi_- (g)$. Unfixing $n$
we see that the set $\{ n : a_n^g \neq \pi (a_n^g) \in \varphi_- (g) \}$ is finite, and $\varphi_- (g) S \pi$ fails, as required. 
This completes the proof of the theorem.
\end{proof}



\section{Variations on the density number: ZFC results}


Inspired by the variants of the rearrangement number considered in~\cite{BBBHHL20} -- and also by $\rrr '$ of the previous section -- 
we now look at natural variants of the density number. A further motivation for studying these variants comes from an analysis of the proof
of Theorem~\ref{density-nonmeager}; in fact, we shall use the framework we will develop to see what this proof really shows (Theorem~\ref{density-nonmeager2}).

Let $\osc$ be a symbol denoting ``oscillation". For a (necessarily infinite-coinfinite) set $A \subset \omega$, say that $d (A) = \osc$ if $\underline{d} (A) < 
\overline{d} (A)$. Let $\all = [0,1] \cup \{ \osc \}$. This is the set of all possible densities of (infinite-coinfinite) subsets $A$ of $\omega$, where we
include the possibility that the density of $A$ is undefined. We are ready for the main definition of this section.

\begin{definition}   \label{general-density-number}
Assume $X , Y \sub \all$ are such that $X \neq \emptyset$ and for all $x \in X$ there is $y \in Y$ with $y \neq x$ (equivalently either $|Y| \geq 2$ or $Y \sem X
\neq \emptyset$). The {\em $(X,Y)$-density number} $\ddd_{X,Y}$ is the smallest cardinality of a family $\Pi \sub \Sym (\omega)$ such that for every 
infinite-coinfinite set $A \sub \omega$ with $d(A) \in X$ there is $\pi \in \Pi$ such that $d (\pi [A]) \in Y$ and $d (\pi [A]) \neq d(A)$.
\end{definition}

Note in particular that if $d(A) = \osc$ we require that $d(\pi [A]) $ has density. Clearly, if $X' \sub X$ and $Y' \supseteq Y$ then $\ddd_{X' , Y'} \leq \ddd _{X,Y}$.
Also $\ddd = \ddd_{[0,1], \all}$. 

Let us first provide the Tukey framework for $\ddd_{X,Y}$: for $X,Y$ as in Definition~\ref{general-density-number}, consider triples $(D_X ,\Sym (\omega), R_Y)$
where $D_X$ is the collection of infinite-coinfinite sets $A$ with $d(A) \in X$, the relation $R_Y$ is given by $A R_Y \pi$ if $d (\pi [A]) \in Y$
and $d(\pi [A] ) \neq d(A)$, for $A \in D_X$ and $\pi \in \Sym (\omega)$. Then $\ddd_{X,Y} = \dd (D_X, \Sym (\omega), R_Y)$. Let $\ddd_{X,Y}^\perp : =
\bb (D_X, \Sym (\omega), R_Y)$. We are ready to rephrase the main result of the last section in this framework:

\begin{theorem} \label{density-nonmeager2} 
\begin{enumerate}
\item If $\osc \notin X$ and $\osc \in Y$, then $(D_X, \Sym (\omega), R_Y) \leq_T {(\M, \twoom, \not\ni )}$.
\item If $0 \in X $ or $1 \in X$, then $(\omom, \Phi_h, \in^\infty) \leq_T (D_X , \Sym (\omega) , R_Y)$ for $h \in\omom$ with $| h(n) | \geq 2^n +n+1$ for all $n$.
\end{enumerate}
\end{theorem}

\begin{proof}
This is immediate from the proof of Theorem~\ref{density-nonmeager}.
\end{proof}

We shall see later that the assumptions in both parts of this theorem are necessary. For (1), this follows from Theorem~\ref{density-covmeager} and the consistency
of $\non (\M) < \cov (\M)$ (true in the Cohen model), and for (2), from Theorem~\ref{density-reaping} and the consistency of $\rr < \non (\M)$ (which holds e.g.
in the Blass-Shelah model~\cite{BS87}, \cite[pp. 370]{BJ95}). 

\begin{corollary}    \label{density-nonmeager-cor2}
$\ddd = \ddd_{ \{ 0,1 \}, \all} = \ddd_{ [0,1], \{ \osc \} } = \ddd_{ \{ 0,1 \} , \{ \osc \} } = \non (\M)$, with the dual result holding for the dual cardinals.
\end{corollary}

It is natural to look for other lower and upper bounds of $\ddd_{X,Y}$ for various $X$ and $Y$; in particular we may ask:
\begin{itemize}
\item what if $0,1 \notin X$?
\item what if $\osc \in X$ or $\osc \notin Y$?
\end{itemize}
We will mostly obtain lower bounds, but there is one more upper bound too, Theorem~\ref{density-reaping}. First consider the relational system $(\twoom , \N , \in)$. 
Then $\dd (\twoom , \N , \in) = \cov (\N)$ and $\bb (\twoom , \N , \in) = \non (\N)$.

\begin{theorem}   \label{density-covnull}
If $X \cap [0,1] \neq \emptyset$, then $(\twoom , \N , \in) \leq_T (D_X , \Sym (\omega) , R_Y)$.
\end{theorem}

We do not know whether this is also true if $X = \{ \osc \}$ (Question~\ref{osc-covnull}).

\begin{corollary}   \label{density-covnull-cor}
If $X \cap [0,1] \neq \emptyset$, then $\cov (\N) \leq \ddd_{X,Y}$ and $\ddd_{X,Y}^\perp \leq \non (\N)$.
\end{corollary}

\begin{proof}[Proof of Theorem~\ref{density-covnull}]

If $0 \in X$ or $1 \in X$ this follows from Theorem~\ref{density-nonmeager2} because it is well-known and easy to see 
that ${(\twoom , \N , \in)} \leq_T (\omom, \Phi_h, \in^\infty)$ (where $h$ is as in Theorem~\ref{density-nonmeager2}). Indeed, let us show that
${(\omom , \N , \in)} \leq_T (\omom, \Phi_h, \in^\infty)$, where the measure on $\omom$ is the product $\prod_n\mu_n$, and $\mu_n$ is defined as follows: For each $n\in\omega$, fix a series
$\sum_k a^n_k=1$ such that, for all $k\in\omega$, $0<a^n_k<\frac{1}{2^{n+1}h(n)}$, and put
$\mu_n(A)=\sum_{k\in A}a^n_k$, for $A\subseteq\omega$.
 Let $\varphi_-:\omom\to\omom$ be the identity
and $\varphi_+(\phi)=\{x\in\omom\mid x\in^\infty\phi\}$, for $\phi\in\Phi_h$. (Note that $\varphi_+(\phi)$ is a null set) Now, the implication
$x\in^\infty\phi\implies x\in\varphi_+(\phi)$ is trivial, and shows that ${(\omom , \N , \in)} \leq_T (\omom, \Phi_h, \in^\infty)$, as desired.

So assume $r \in (0,1) \cap X$. We need to define functions $\varphi_- : 2^\omega \to D_X$ and $\varphi_+ : \Sym(\omega) \to \mathcal{N}$ such that 
for all $x \in 2^\omega$ and $\pi \in \Pi_d$, if $\varphi_- (x) R_r \pi$ then $x \in \varphi_+ (\pi)$. Consider the product measure $\mu$ on $2^\omega$ of the measure $m$ with $m ( \{ 1 \} ) = r$ and $m ( \{ 0 \} ) = 1 -r$. Let $X_n: 2^\omega \to 2$ be the random variable with $X_n (x) = x (n)$. 
Clearly the expected value of all $X_n$ is $r$. By the strong law of large numbers we see 
$$\mu ( \{ x \in 2^\omega : d(x) = r \} ) = \mu \left( \lim_{n \to \infty} {1 \over n} ( X_0 + ... + X_{n-1}) = r \right) = 1.$$
Also, if $\pi$ is a permutation, then $$\mu ( \{ x \in 2^\omega : d (\pi (x)) = r \} ) = \mu ( \pi^{-1} \{ x \in 2^\omega : d(x) = r \} ) = 1.$$

Now, let $$\varphi_- (x) = x \mbox{ if } x \in D_{ \{r\} }$$
($\varphi_-(x)$ is irrelevant if $x \notin D_{r}$) and 
$$\varphi_+ (\pi) = \{ y \in 2^\omega : d(y) \neq r \mbox{ or } d (\pi(y)) \neq r \}.$$ By the above discussion we see that $\mu (\varphi_+ (\pi)) = 0$ as required.

Let $x \in 2^\omega$ and $\pi \in \Sym (\omega)$, and assume $d(\pi(\varphi_- (x))) \neq r$. Then either $x \notin D_{ \{r\} }$ and $d(x) \neq r$ or $\varphi_- (x) = x$ and $d(\pi(x)) \neq r$. 
So $x \in \varphi_+ (\pi)$. This completes the proof of the theorem.
\end{proof}

Next recall that for functions $f,g \in\omom$, $f \leq^* g$ ($f$ is {\em eventually dominated by} $g$) if $f (n) \leq g(n)$ holds for all but finitely many $n
\in \omega$, and consider the relational system $(\omom, \omom, \not\geq^*)$. Then $\dd (\omom, \omom, \not\geq^*)$ is the unbounding number $\bb$ and
$\bb (\omom, \omom, \not\geq^*)$ is the dominating number $\dd$. We show:

\begin{theorem}   \label{density-unbounding}
If $\osc \in X$ or $\osc \notin Y$, then $(\omom,\omom, \not\geq^*) \leq_T (D_X , \Sym (\omega) , R_Y)$.
\end{theorem}

We note that since $(\omom, \omom, \not\geq^* ) \leq_T (\omom, \Phi_h , \in^\infty)$, this is also true if $0\in X$ or $1 \in X$, by Theorem~\ref{density-nonmeager2}.
However, it will follow from Theorem~\ref{density-lessthanb} that this (consistently) fails if $X \sub (0,1)$ and $\osc \in Y$. 

\begin{corollary}   \label{density-unbounding-cor}
If $\osc \in X$ or $\osc \not\in Y$, then $\bb \leq \ddd_{X,Y}$ and $\ddd_{X,Y}^\perp \leq \dd$.
\end{corollary}

\begin{proof}[Proof of Theorem~\ref{density-unbounding}]
We need to define $\varphi_- : \omom \to D_X$ and $\varphi_+ : \Sym (\omega) \to \omom$ such that for all $g \in \omom$ and $\pi \in \Sym (\omega)$, if $\varphi_- (g) R_Y \pi$
then $g \not\geq^* \varphi_+ (\pi)$. First assume $\osc \in X$.

Let $\varphi_+ (\pi) (n) =  \max \{ \pi (k), \pi^{-1} (k) : k \leq n \} + 2^n$ for $\pi \in \Sym (\omega)$.
Assume $g(n)\geq 2^n$ for almost all $n$. Define $\{ i^g_n : n \in \omega \}$ as follows: let $n_0$ be minimal such that $g(n) \geq 2^n$ for all $n \geq n_0$.
Let $i^g_0 = 0$ and
\[ i_{n+1}^g = \left\{ \begin{array}{ll} 2^{i^g_n} & \mbox{ if } n < n_0 \\
   g(i^g_n) & \mbox {if } n \geq n_0 \end{array} \right. \]
Finally let $\varphi_- (g) = \bigcup_{n \in\omega} [ i^g_{4n} , i^g_{4n + 2} )$. It is easy to see that $d (\varphi_- (g)) = \osc$; in fact, $\underline{d} (\varphi_- (g)) = 0$
and $\overline{d} (\varphi_- (g)) =1$. (For other $g$ the definition is irrelevant.)

Assume $g \in \omom$ and $\pi \in \Sym (\omega)$ are such that $g \geq^* \varphi_+ (\pi)$. Then $g (n) \geq 2^n$ for almost all $n$ so that the above definition
of $\varphi_- (g)$ applies. Let $n_0$ be such that $g(n) \geq \varphi_+ (\pi ) (n)$ for all $n \geq n_0$. Fix such $n$. Note that if $k < i_{4n}^g$ then $\pi (k)
< \varphi_+ (\pi) (k) < \varphi_+ (\pi) (i^g_{4n}) \leq g (i^g_{4n}) = i_{4n +1}^g$. Hence $\pi (k) \in i_{4n+1}^g$. The same argument shows that if $k < i^g_{4n + 1}$ then
$\pi^{-1} (k) \in i_{4n + 2}^g$, that is, if $k \geq i_{4n + 2}$ then $\pi (k) \geq i^g_{4n+1}$. Therefore it follows that $| \pi [ \varphi_- (g) ] \cap i_{4n + 1}^g |
= | \varphi_- (g) \cap i^g_{4n + 1} |$, and this value converges to $1$. Similarly, $| \pi [ \varphi_- (g) ] \cap i_{4n + 3}^g |
= | \varphi_- (g) \cap i^g_{4n + 3} |$, which converges to $0$. Thus we still have $\underline{d} ( \pi [ \varphi_- (g) ] ) = 0$ and $\overline{d}
(\pi [ \varphi_- (g) ] ) = 1$, that is, $d (\pi [ \varphi_- (g)]) = \osc$, and $\varphi_- (g) R_Y \pi$ fails, as required.

Next assume $\osc \notin Y$. Note that by the first half of the proof  and the comment after the statement of the theorem, there is nothing to show if $\osc \in X$
or $0 \in X$ or $1 \in X$. Hence assume $X \sub ( 0,1 )$. To illustrate the main idea, first consider the case ${1 \over 2} \in X$. 

\underline{Case 1. $r = {1 \over 2}$.} Let $E$ be the even numbers, and $O$, the odd numbers. Say $\pi \in\Sym (\omega)$ is {\em big on $E$} if there is
$k = k_\pi$ such that there are infinitely many $m$ such that 
\begin{equation} { | \pi [ E ] \cap m | \over m } > {1\over 2 } + {1 \over k} \label{eqn1}   \end{equation}
Similarly, $\pi \in\Sym (\omega)$ is {\em small on $E$} if there is
$k = k_\pi$ such that there are infinitely many $m$ such that 
\begin{equation} { | \pi [ E ] \cap m | \over m } < {1\over 2 } - {1 \over k} \label{eqn2}   \end{equation}
If $\pi$ is big on $E$ (small on $E$, respectively), define 
\[ m^\pi_n = \min \left\{ m \geq 2^n : m \mbox{ satisfies (\ref{eqn1}) } (m \mbox{ satisfies (\ref{eqn2}), respectively}) \right\} \]
If $\pi$ is neither big nor small on $E$, define 
\[ m^\pi_n =  \min \left\{ m \geq 2^n :  {1\over 2} - {1\over 2^n}  \leq  { | \pi [ E ] \cap m | \over m } \leq {1\over 2} + {1 \over 2^n} \right\} \]
Finally, for any $\pi \in \Sym (\omega)$, let
\[ \varphi_+ (\pi)(n) = \min \left\{ u > m^\pi_n : \pi^{-1} [ m^\pi_n] \sub u \right\} \]

Next assume $g(n) \geq 2^n$ for almost all $n$. Define $\{ i^g_n : n \in \omega \}$ as follows: let $n_0$ be minimal such that $g(n) \geq 2^n$ for all
$n \geq n_0$. Let $i^g_0 = 0$ and
\[ i^g_{n+1} = \left\{ \begin{array}{ll} 2^{i^g_n} & \mbox{ if } n < n_0 \\
   g( i^g_n) & \mbox{ if } n \geq n_0 \end{array} \right. \]
Define 
\[ \varphi_- (g) = \bigcup_{n\in\omega} \left( \left[ i^g_{2n}, i^g_{2n+1} \right) \cap E \right) \cup \bigcup_{n\in\omega} \left( \left[ i^g_{2+1} , i^g_{2n+2} \right) \cap O \right) \]
Clearly $d (\varphi_- (g) ) = {1 \over 2}$.

Assume $g \in\omom$ and $\pi \in \Sym (\omega)$ are such that $g \geq^* \varphi_+ (\pi)$. We need to show that either $d (\pi [ \varphi_- (g)]) = {1 \over 2}$ or
$d (\pi [ \varphi_- (g)]) = \osc \notin Y$: the point is that $\varphi_- (g) R_Y \pi$ fails in both cases. Clearly $g (n) \geq 2^n$ for almost all $n$, and the above definition
of $\varphi_- (g)$ applies. Let $n_0$ be such that $g(n) \geq \varphi_+  (\pi) (n)$ for all $n \geq n_0$.
 
First assume $\pi$ is big on $E$. Fix $2n \geq n_0$. Then 
\[ 2^{i^g_{2n}} \stackrel{(1)}{ \leq} m_{i^g_{2n}}^\pi  < \varphi_+ (\pi) (i^g_{2n}) \stackrel{(2)}{\leq} g(i^g_{2n}) = i^g_{2n+1} \]
and therefore, by (2) and~(\ref{eqn1}),
\[ {\left | \pi [E \cap i^g_{2n+1} ]\cap m^\pi_{i^g_{2n}} \right| \over m^\pi_{i^g_{2n}} } =  { \left| \pi [E  ]\cap m^\pi_{i^g_{2n}} \right| \over m^\pi_{i^g_{2n}} } > {1 \over 2} + {1 \over k_\pi} \]
By (1) and definition of $\varphi_- (g)$, we see that for large enough $n$,
\[  { \left| \pi [  \varphi_- (g) ]\cap m^\pi_{i^g_{2n}} \right| \over m^\pi_{i^g_{2n}} } > {1 \over 2} + {1 \over 2 k_\pi} \]
For the same reason, for large enough $n$,
\[  { \left| \pi [  \varphi_- (g) ]\cap m^\pi_{i^g_{2n + 1}} \right| \over m^\pi_{i^g_{2n + 1}} } < {1 \over 2} - {1 \over 2 k_\pi} \]
Thus $d (\pi [ \varphi_- (g)]) = \osc$, as required. The proof in case $\pi$ is small on $E$ is analogous.

Finally assume $\pi$ is neither big nor small on $E$. In this case it is easy to see that the numbers
\[ { \left| \pi [E \cap i^g_{n+1} ]\cap m^\pi_{i^g_{n}} \right| \over m^\pi_{i^g_{n}} } =  { \left| \pi [E  ]\cap m^\pi_{i^g_{n}} \right| \over m^\pi_{i^g_{n}} } \]
converge to ${1 \over 2}$, and therefore the same is true for the numbers
\[  { \left| \pi [  \varphi_- (g) ]\cap m^\pi_{i^g_{n}} \right| \over m^\pi_{i^g_{n}} } \]
Hence either $d (\pi [ \varphi_- (g)]) = {1 \over 2}$ or $d (\pi [ \varphi_- (g)]) = \osc$, as required, and we are again done.

\underline{Case 2. Arbitrary $r \in (0,1)$.}   
Let $(\ell_b : b \in \omega)$ be a sequence such that ${\ell_b \over 2^b}$ converges to $r$. We may assume that $| r - {\ell_b \over 2^b}| <
{1 \over 2^b}$. Recursively define $(j_b : b \in \omega)$ such that
$j_0 = 0$ and $j_{b+1} = j_b + 2^{2b}$. Partition the interval $J_b = [j_b, j_{b+1})$ into $2^b$ many intervals $I_b^a$, $a \in 2^b$, of length $2^b$.
For $c \in \omega$, define $w_c := \left( \begin{array}{c} 2^c \\ \ell_c \end{array}  \right)$. 
Note that $w_c$ is the number of $\ell_c$-element subsets of $2^c$. Let $e_c : w_c \to \P (2^c)$ be a bijection from $w_c$ to the $\ell_c$-element subsets
of $2^c$ such that $e_c (0) = \ell_c$, that is, $e_c (0)$ consists of the first $\ell_c$ elements of $2^c$. Next, for $b \geq c$ and $a \in 2^b$, partition $I^a_b$ into $2^c$ many intervals 
$(L^{a,c}_{b,d} : d < 2^c  )$ of length $2^{b -c}$. For $v \in w_c$, define $E^c_v \sub \omega$ such that $E^c_v \cap j_c = \emptyset$ and 
\[   E^c_v \cap I^a_b = \bigcup \{ L^{a,c}_{b,d} : d \in e_c (v) \} \]
for $b \geq c$ and $a \in 2^b$. Note that $d (E^c_v) = {\ell_c \over 2^c}$. 
Similar to Case 1,  say $\pi \in \Sym (\omega)$ is {\em big} ({\em small}, respectively) if there is $k = k_\pi$ such that for every $c_0$ there is $c \geq c_0$ such
that 
\begin{equation} \exists^\infty m \;\;\; { | \pi [E^c_0] \cap m | \over m } > {\ell_c \over 2^c}  + {1\over k} \left( < {\ell_c \over 2^c} - {1 \over k} \mbox{ respectively}\right) 
\label{eqn3}  \end{equation}
Note that this means (in the big case) that for such $c$ we can find $v(c) \in w_c$ such that for infinitely many $m$,
\begin{equation} { | \pi [E^c_0] \cap m | \over m } > {\ell_c \over 2^c}  + {1\over k} \mbox{ and } { \left| \pi [E^c_{v(c)} ] \cap m \right| \over m } < {\ell_c \over 2^c} 
\label{eqn4} \end{equation}
and similarly for the small case. If $\pi$ is big (or small) and $c < n$ satisfies~(\ref{eqn3}), define 
\[ m^\pi_{n,c} = \min \{ m \geq 2^n : m \mbox{ satisfies (\ref{eqn4})} \} \]
and let
\[  m^\pi_n = \max \{ m^\pi_{n,c} : c < n \mbox{ and } c \mbox{ satisfies~(\ref{eqn3})} \} \]
If $\pi$ is neither big nor small, there is a sequence $(k_c : c \in \omega)$ going to $\infty$ such that~(\ref{eqn3}) fails for $c$ with $k = k_c$. Then define
\begin{equation} m^\pi_n =  \min \left\{ m \geq 2^n :  {\ell_c \over 2^c} - {1\over k_c}  \leq  { | \pi [ E^c_0 ] \cap m | \over m } \leq {\ell_c \over 2^c} + {1 \over k_c} \right\} 
\label{eqn5}  \end{equation}
Finally, for $\pi \in \Sym (\omega)$ and $n \in \omega$,  $\varphi_+ (\pi) (n) := \min \left\{ u > m^\pi_n : \pi^{-1} [ m^\pi_n] \sub u \right\}$ is defined exactly as in the special case.


Next assume that $g(n) \geq 2^n$ for almost all $n$. Define $\{ i^g_n : n \in \omega \}$ as follows: let $n_0$ be minimal such that $g(n) \geq 2^n$ for all
$n \geq n_0$. Let $i^g_0 = 0$ and
\[ i^g_{n+1} = \left\{ \begin{array}{ll} 2^{i^g_n} & \mbox{ if } n < n_0 \\
   g( i^g_n) & \mbox{ if } n \geq n_0 \end{array} \right. \]
Let $( K_c : c \in \omega )$ be the interval partition of $\omega$ into intervals of length $w_c$. If $n$ is the $v$-th element of $K_c$
($v < w_c$), $i^g_n \leq b < i^g_{n + 1}$, and $a \in 2^b$, 
define \[ \varphi_- (g) \cap I^a_b = \bigcup \{ L^{a,c}_{b,d} : d \in e_c (v) \} \]
Clearly $d (\varphi_- (g)) = r$.  Also note that $\varphi_- (g) \cap [i^g_n , i^g_{n+1} ) = E^c_v \cap [i^g_n , i^g_{n+1} ) $ $(\star)$.


Given $g \in\omom$ and $\pi \in \Sym (\omega)$ with $g \geq^* \varphi_+ (\pi)$, we need to show again that either $d (\pi [ \varphi_- (g)]) = r$ or
$d (\pi [ \varphi_- (g)]) = \osc \notin Y$. Let $n_0$ be such that $g(n) \geq \varphi_+  (\pi) (n)$ for all $n \geq n_0$. Let $n(c)$ be the $0$-th element of
$K_c$ for $c \in \omega$. Assume $n(c) \geq n_0$. In case $\pi$ is neither
big nor small, we see by~(\ref{eqn5}) that the numbers
\[ { \left| \pi [E^c_0 \cap i^g_{n(c) +1} ]\cap m^\pi_{i^g_{n(c)}} \right| \over m^\pi_{i^g_{n(c)}} } =  { \left| \pi [E^c_0  ]\cap m^\pi_{i^g_{n(c)}} \right| \over m^\pi_{i^g_{n(c)}} } \]
converge to $r$ as $c$ goes to infinity, and therefore, by $(\star)$, so do the numbers
\[ \left| \pi [ \varphi_- (g) ] \cap m^\pi_{i^g_{n(c)}} \right| \over m^\pi_{i^g_{n(c)}} \]
Thus either $d (\pi [ \varphi_- (g)]) = r$ or $d (\pi [ \varphi_- (g)]) = \osc $ as required. So assume without loss
of generality that $\pi$ is big. (The proof in case $\pi$ is small is analogous.) 
For $n(c) \geq n_0$, we see by~(\ref{eqn4}) that 
\[ { \left| \pi [E^c_0 \cap i^g_{n (c) +1} ]\cap m^\pi_{i^g_{n (c)}} \right| \over m^\pi_{i^g_{n (c)}} }  =  { \left| \pi [E^c_0  ]\cap m^\pi_{i^g_{n(c)}} \right| \over m^\pi_{i^g_{n(c)}} } 
> {\ell_c \over 2^c}   + {1 \over k_\pi} \]
and, using additionally $(\star)$, we infer that $\overline{d} (\pi [ \varphi_- (g)]) \geq r + {1 \over k_\pi}$. On the other hand, 
letting $n'(c)$ be the $v(c)$-th element of $K_c$ for $c \in \omega$, (\ref{eqn4}) also implies that 
\[ { \left| \pi [E^c_{v(c)} \cap i^g_{n' (c) +1} ]\cap m^\pi_{i^g_{n' (c)}} \right| \over m^\pi_{i^g_{n '(c)}} }  =  { \left| \pi [E^c_{v(c)}  ]\cap m^\pi_{i^g_{n' (c)}} \right| \over m^\pi_{i^g_{n' (c)}} } 
< {\ell_c \over 2^c}   \]
so that, by $(\star)$,  $\underline{d} (\pi [ \varphi_- (g)]) \leq r $.
Thus $d (\pi [ \varphi_- (g)]) = \osc $, and the proof of the theorem is complete.
\end{proof}

Next consider the relational system $(\twoom , \M , \in)$. Note that this is the dual of the system considered in Section 2, $(\twoom, \M ,\in) = (\M ,\twoom ,\not\ni)^\perp$.
In particular, $\dd (\twoom , \M ,\in) = \cov (\M)$ and $\bb (\twoom , \M,  \in) = \non (\M)$.

\begin{theorem}    \label{density-covmeager}
If $\osc \in X$ or $\osc \notin Y$, then $(\twoom, \M , \in) \leq_T (D_X , \Sym (\omega), R_Y)$.
\end{theorem}

Note that -- as for Theorem~\ref{density-unbounding} -- the assumption is optimal because if $\osc \notin X$ and $\osc \in Y$ we are in the situation of the first half of 
Theorem~\ref{density-nonmeager2} and $(D_X, \Sym(\omega),R_Y) \leq_T (\M ,\twoom ,\not\ni)$ holds.

\begin{corollary}   \label{density-covmeager-cor}
If $\osc \in X$ or $\osc \notin Y$, then $\cov (\M) \leq \ddd_{X,Y}$ and $\ddd_{X,Y}^\perp \leq \non (\M)$.
\end{corollary}

\begin{proof}[Proof of Theorem~\ref{density-covmeager}] We need to define $\varphi_- : \twoom \to D_X$ and $\varphi_+ : \Sym (\omega) \to \M$ such that
for all $x \in \twoom$ and $\pi \in \Sym (\omega)$, if $\varphi_- (x) R_Y \pi$ then $x \in \varphi_+ (\pi)$.

First consider the case $\osc \in  X$. Fix arbitrary $y \in D_{ \{ \osc \} }$. Let
\[ \varphi_- (x) = \left\{ \begin{array}{ll} x & \mbox{ if } x \in D_{ \{ \osc \} } \\ 
   y & \mbox{ if } x \notin D_{ \{ \osc \} }  \end{array} \right. \]
for $x \in \twoom$ and
\[ \varphi_+ (\pi) = \{ x \in \twoom : d(x) \neq \osc \mbox{ or } d ( \pi [x]) \neq \osc \} \]
for $\pi \in \Sym (\omega)$. Then $\varphi_+ (\pi) \in \M$   because $D_{ \{ \osc \} }$ is comeager in $\twoom$ and so is $\pi^{-1} [ D_{ \{ \osc \} } ]$.

Let $x \in \twoom$ and $\pi \in \Sym (\omega)$. Then either $d(x) \neq \osc$ and $x \in \varphi_+ (\pi)$ follows, or $d(x) = \osc$,
$\varphi_- (x) = x$, and if $x R_Y \pi$ then $d ( \pi [x] ) \neq \osc$ and $x \in \varphi_+ (\pi)$ follows again.

We now come to the case $\osc \notin Y$. If $\osc \in X$, we are done by the above argument, and if $0 \in X$ or $1 \in X$,
apply Theorem~\ref{density-continuum} below (using the trivial $(\twoom , \M, \in) \leq_T (\cc, [\cc]^{\aleph_0} , \in)$).
Hence it suffices to consider the case $(0,1) \cap X \neq \emptyset$. To illustrate the basic idea, we first deal with the special
case ${ 1 \over 2} \in X$.

\underline{Case 1. \textit{$r = {1 \over 2}$}}. (Special case) Define $\varphi_-$ such that for $x \in 2^\omega$, 
$$\varphi_- (x) (2n) = \left\{ \begin{array}{ll} 0 & \mbox{ if } x(n) = 0 \\ 1 & \mbox{ if } x(n) = 1 \end{array} \right. \;\;\;\mbox{ and } \;\;\; \varphi_- (x) (2n+1) = \left\{ \begin{array}{ll} 1 & \mbox{ if } x(n) = 0 \\ 0 & \mbox{ if } x(n) = 1 \end{array} \right.$$ 
It is clear that $d (\varphi_-(x) ) = {1\over 2}$. Next let 
$$\varphi_+ (\pi) = \left\{ x : d (\pi[\varphi_- (x)]) \mbox{ is distinct from }  {1\over 2} \mbox{ and } \osc \right\}$$ 
for $\pi \in \Sym (\omega)$. The conclusion then holds obviously, and we need to show that $\varphi_+ (\pi) \in \mathcal{M}$. 
Clearly it suffices to prove that the set $$\left\{ x : \overline{d} (\pi[\varphi_- (x)]) \geq {1 \over 2} \mbox{ and } \underline{d} (\pi [\varphi_- (x)]) \leq {1 \over 2} \right\}$$ is comeager.

Given $s \in 2^{< \omega}$ we define $\varphi_- (s)$ as above with $|\varphi_- (s)| = 2 |s|$. 
(So this means that for any $x \in 2^\omega$ with $s \subseteq x$, $\varphi_- (s) = \varphi_- (x) \restriction 2 |s|$.) 
We claim that given $s \in 2^{< \omega}$ and $n_0 \in \omega$ there are $n \geq n_0$ and $t \supseteq s$ with $\pi^{-1} [2n] \subseteq 2 |t|$ and such that 
$${| \pi[\varphi_- (t)] \cap 2 n | \over 2n} \geq {1 \over 2}$$ 
and similarly with $\geq$ replaced by $\leq$. This claim clearly finishes the proof.

To see the claim (for $\geq$), assume without loss of generality that $|s| = n_0$. Choose $n \geq n_0$ such that $\pi [2 n_0] \subseteq 2n$, 
and then choose $m \geq n$ such that $\pi^{-1} [2n] \subseteq 2m$. Extend $s$ to $t$ with $|t| = m$ as follows. 
Recursively construct $s_j$ for $j$ with $n_0 \leq j \leq m$ such that, letting 
$$c_j = n_0 + | \{ 2n_0 \leq i < 2j : \pi (i) \in 2n \mbox{ and } \varphi_- (s_j) (i) = 1 \} | $$ 
and 
$$ d_j = 2 n_0 + | \{ 2n_0 \leq i < 2j : \pi (i) \in 2n  \} |$$ 
we have
    \begin{itemize}
     \item $|s_j| = j$,
        \item $s = s_{n_0} \subset s_{n_0 + 1} \subset ... \subset s_j \subset ... \subset s_m = t$,
        \item ${c_j \over d_j} \geq {1 \over 2}$.
    \end{itemize}
Suppose $s_j$ has been constructed. If either (case 1) both $\pi (2j)$ and $\pi(2j + 1)$ belong to $2n$ or (case 2) both do not belong to $2n$, 
we extend $s_j$ arbitrarily to $s_{j+1}$ (i.e. it does not matter whether $s_{j+1} (j)$ is $0$ or $1$). 
If (case 3) $\pi (2j) \in 2n$ and $\pi (2j + 1) \notin 2n$, we let $s_{j+1} (j) =1$, and if (case 4) $\pi (2j) \notin 2n$ and $\pi (2j + 1) \in 2n$, we let $s_{j+1} (j) =0$.

We need to check the last item, that is, ${c_{j+1} \over d_{j+1}} \geq {1 \over 2}$. By induction hypothesis, we know ${c_j \over d_j} \geq {1\over 2}$. 
Let $c_j = {d_j \over 2} + k_j$ with $k_j \geq 0$. In case 1 we have $c_{j+1} = c_j + 1$ and $d_{j+1} = d_j + 2$. Thus 
$${c_{j+1} \over d_{j+1}} = { {d_j + 2 \over 2 } + k_j \over d_j + 2} = {1 \over 2} + {k_j \over d_j + 2} \geq {1\over 2} $$ 
as required. In case 2, $c_{j+1} = c_j$ and $d_{j+1} = d_j$ and the conclusion is trivial. In cases 3 and 4, $c_{j+1} = c_j + 1$ and $d_{j+1} = d_j + 1$ and we get 
$${c_{j+1} \over d_{j+1}} = { {d_j + 1 \over 2 } + k_j + {1 \over 2} \over d_j + 1} = {1 \over 2} + {k_j + {1 \over 2} \over d_j + 1} \geq {1\over 2}. $$ 
This completes the construction.

Now simply note that for $j = m$, the assumptions $\pi [2 n_0] \subseteq 2n$ and $\pi^{-1} [2n] \subseteq 2m$ imply that $d_m = 2n$ and $c_m = |\pi [\varphi_- (t)] \cap 2n|$. Hence we obtain 
$${| \pi[\varphi_- (t)] \cap 2 n | \over 2n} \geq {1 \over 2}$$ 
as required.

\noindent \underline{Case 2. Arbitrary \textit{$r \in (0,1)$}}. (General case) 
As in the proof of Theorem~\ref{density-unbounding}, let $(\ell_b : b \in \omega)$ be a sequence such that ${\ell_b \over 2^b}$ converges to $r$ and $| r - {\ell_b \over 2^b}| <
{1 \over 2^b}$. Let $(i_b : b \in \omega)$ be such that $i_0 = 0$ and $i_{b+1} = i_b + 2^b$, and let $(j_b : b \in \omega)$ such that
$j_0 = 0$ and $j_{b+1} = j_b + 2^{2b}$. Partition the interval $J_b = [j_b, j_{b+1})$ into $2^b$ many intervals $I_b^a$, $a \in 2^b$, of length $2^b$.
Again let $w_b := \left( \begin{array}{c} 2^b \\ \ell_b \end{array}  \right)$, the number of $\ell_b$-element subsets of $2^b$, and
let $e_b : w_b \to \P (2^b)$ be a bijection from $w_b$ to the $\ell_b$-element subsets of $2^b$. Let $f \in\omom$ be the function with $f(n) = w_b$ whenever 
$i_b \leq n < i_{b+1}$. Work with the space of functions below $f$, $\prod f := \prod_{n\in\omega} f(n)$, instead of $\twoom$ (this is possible because
$(\twoom, \M ,\in) \equiv_T (\prod f, \M ,\in )$). For $x \in \prod f$ define $\varphi_- (x) \in \omoms$ such that for all $b \in \omega$ and all $a \in 2^b$,
\[ \varphi_- (x) \cap I^a_b = g^a_b [ e_b (x (i_b + a))] \]
where $g^a_b$ is the canonical bijection between $2^b$ and $I^a_b$ sending $k \in 2^b$ to the $k$-th element of $I^a_b$. (Note here that
$x (i_b + a) \in f (i_b + a) = w_b$.) It is clear that $d (\varphi_- (x)) = r$. Let
\[ \varphi_+ (\pi) = \left \{ x: d ( \pi [ \varphi_- (x)]) \mbox{ is distinct from } r \mbox{ and } \osc \right\} \]
for $\pi \in \Sym (\omega)$. Again, we only need to show that $\varphi_+ (\pi) \in \M$ and, again, this is done by proving that the set
\[ \left\{ x : \overline{d} ( \pi [ \varphi_- (x) ]) \geq r \mbox{ and } \underline{d} ( \pi [ \varphi_- (x) ]) \leq r \right\} \]
is comeager.

Given $s \in \prod^\leq f := \bigcup_{n \in\omega} \prod_{k < n} f(k)$ define $\varphi_- (s)$ as above. 
We claim that given $s \in \prod^\leq f$, $b_0 \in \omega$ and $\epsilon > 0$ there are $b_2 \geq b_1 \geq b_0$ and $t \supseteq s$ with
$t \in \prod^\leq f$ and $ | t | = i_{b_2}$ such that $\pi^{-1} [ j_{b_1} ] \sub j_{b_2}$ and
\[ { \left| \pi [ \varphi_- (t) ] \cap j_{b_1} \right| \over j_{b_1} } \geq r - \epsilon \]
and similarly with $\geq r - \epsilon$ replaced by $\leq r + \epsilon$. This will finish the proof.

To see the claim (for $\geq r - \epsilon$), assume without loss of generality that $|s| = i_{b_0}$. Also assume $b_0$ is so large that
\[ { |  \varphi_- (s) \cap j_{b_0} | \over j_{b_0} } \geq r - \epsilon \mbox{ and } {\ell_b \over 2^b } \geq r - \epsilon \mbox{ for all } b \geq b_0 \]
Choose $b_1 \geq b_0$ such that $\pi [ j_{b_0} ] \sub j_{b_1}$, and then choose $b_2 \geq b_1$ such that $\pi^{-1} [ j_{b_1} ] \sub j_{b_2}$.
Extend $s$ to $t$ with $|t| = i_{b_2}$ as follows. Recursively construct $s_k$ for $k$ with $i_{b_0} \leq k \leq i_{b_2}$ such that if
$k = i_b + a$ for some $b$ with $b_0 \leq b \leq b_2$ and $a \in 2^b$, letting
\[ c_k = | \varphi_- (s) \cap j_{b_0} | + | \{ j_{b_0} \leq \ell < j_b + a \cdot 2^b : \pi (\ell) \in j_{b_1} \mbox{ and } \ell \in \varphi_- (s_k) \} | \]
and
\[ d_k = j_{b_0} + | \{ j_{b_0} \leq \ell < j_b + a \cdot 2^b : \pi (\ell) \in j_{b_1} \} | \]
we have
\begin{itemize}
\item $| s_k | = k$,
\item $s = s_{i_{b_0}} \subset s_{i_{b_0} + 1} \subset ... \subset s_k \subset ... \subset s_{i_{b_2}} = t$,
\item ${ c_k \over d_k } \geq r - \epsilon$.
\end{itemize}
Note that the last item holds for $k = i_{b_0}$ by assumption. Suppose $s_k$ has been constructed for some $k = i_b + a$ with $b_0 \leq b < b_2$ and $a \in 2^b$.
We need to define $s_{k+1} (k) = s_{k+1} (i_b + a) \in f (i_b + a) = w_b$. Consider the set $z^a_b = \pi^{-1} [ j_{b_1} ] \cap I^a_b$. If this set has at least
$\ell_b$ elements choose any $\ell_b$-element subset $y^a_b$ of $z^a_b$. Otherwise let $y^a_b$ any $\ell_b$-element subset of $I^a_b$ containing $z^a_b$.
Let \[  s_{k+1} (i_b + a) := e_b^{-1} \left( (g^a_b)^{-1} [y^a_b] \right) \in w_b \]
So $\varphi_- (s_{k+1}) \cap I^a_b = y^a_b$. 

We need to check the last item, that is, ${c_{k+1} \over d_{k+1} } \geq r - \epsilon$. Note that, in terms of ${c_k \over d_k}$,  the smallest possible values for this quotient
are obtained \underline{either} if $z^a_b  = I^a_b$ in which case $c_{k+1} = c_k + \ell_b$ and $d_{k+1} = d_k + 2^b$ \underline{or} if $z^a_b = \emptyset$ in which
case $c_{k+1} = c_k$ and $d_{k+1} = d_k$. Therefore we obtain
\[ {c_{k+1} \over d_{k+1} } \geq \min \left\{   {c_k \over d_k} , { c_k + \ell_b \over  d_k + 2^b } \right\} \]
Since the latter value is between ${c_k \over d_k}$ and ${\ell_b \over 2^b}$, and both of these are $\geq r - \epsilon$, the first by inductive assumption and
the second by choice of $b_0$, we see that ${c_{k+1} \over d_{k+1} } \geq r - \epsilon$, as required. This completes the construction.

Now simply note that for $k = i_{b_2}$, the assumptions $\pi [ j_{b_0} ] \sub j_{b_1}$ and $\pi^{-1} [j_{b_1} ] \sub j_{b_2}$ imply
that $d_{i_{b_2}} = j_{b_1}$ and $c_{i_{b_2}} = | \pi [ \varphi_- (t) ] \cap j_{b_1} |$. Therefore we obtain
\[ { \left| \pi [ \varphi_- (t) ] \cap j_{b_1} \right| \over j_{b_1} } \geq r - \epsilon \]
as required. This finishes the proof of the theorem.
\end{proof}

The simple relational system $(\cc, [\cc]^{\aleph_0} , \in )$ satisfies $\dd (\cc, [\cc]^{\aleph_0} , \in) = \cc$ and ${\bb (\cc, [\cc]^{\aleph_0} , \in )} = \aleph_1$.

\begin{theorem} \label{density-continuum}
If $0 \in X$ or $1 \in X$ and $\osc \notin Y$, or if $\osc \in X$ and $0 \notin Y$ or $1 \notin Y$, then $( \cc, [\cc]^{\aleph_0} , \in ) \leq_T (D_X, \Sym (\omega) , R_Y)$.
\end{theorem}

The assumptions in this theorem are mostly necessary: by Theorems~\ref{density-nonmeager2} (1) and~\ref{density-reaping}, if $\osc \in Y$ and $\osc \notin X$, then $\ddd_{X,Y} \leq \non (\M)$,
if $\osc \in Y$ and $0 \in Y$ and $1 \in Y$, then $\ddd_{X,Y} \leq \max \{ \rr, \non (\M) \}$, and if $0\notin X$ and $1\notin X$ and $0 \in Y$ and $1\in Y$, then $\ddd_{X,Y} \leq \rr$.
However, we do not know whether we can have $\ddd_{X,Y} = \cc$ in case $0 \notin X$ and $1\notin X$ and $\osc \notin X$, i.e., in case $X \sub (0,1)$, though
$\ddd_{X,Y} \leq \rr$ will follow in some cases from Theorem~\ref{density-reaping}.

\begin{corollary}   \label{density-continuum-cor}
Under the assumptions of Theorem \ref{density-continuum}, $\ddd_{X,Y} = \cc$ and $\ddd_{X,Y}^\perp \leq \aleph_1$.
\end{corollary}

Note that in general we can only prove $\ddd_{X,Y}^\perp \leq \aleph_1$ even though in natural cases $\ddd_{X,Y}^\perp = \aleph_1$ will hold.
By definition $\ddd_{X,Y}^\perp \geq 2$, and if for example $X = \{ 0 \}$ and $Y = ( {1 \over 2} , 1]$ then $\ddd_{X,Y}^\perp = 2$, as witnessed by two
disjoint sets of density $0$. Similarly, it is easy to see that if $X = \{ 0 \}$ and $Y = ( {1 \over 3} , 1]$, then $\ddd_{X,Y}^\perp = 3$, etc. So all finite numbers $\geq 2$
can be realized as $\ddd_{X,Y}^\perp$. We do not know whether $\ddd_{X,Y}^\perp = \aleph_0$ for some choices of $X$ and $Y$.

\begin{proof}[Proof of Theorem~\ref{density-continuum}] 
We need to define $\varphi_- : \cc \to D_X$ and $\varphi_+ : \Sym (\omega) \to [ \cc]^{\aleph_0}$ such that for all $\alpha \in \cc$ and $\pi \in
\Sym (\omega)$, if $\varphi_- (\alpha) R_Y \pi$ then $\alpha \in \varphi_+ (\pi)$.

First consider the first case. Assume without loss of generality that $0 \in X$. (The case $1 \in X$ is analogous: work with complements!)
Let $A$ be such that $d(A)=0$. Let $\{A_\alpha  :  \alpha < \cc \}$ be an almost disjoint  family of subsets of $A$. So $d(A_\alpha) = 0$ for all $\alpha  < \cc$.
Let $\varphi_- (\alpha) = A_\alpha$. Fix  a permutation $\pi$. Note that $\varphi_+ (\pi) := \{ \alpha : d (\pi [\varphi_- (\alpha)]) $ is defined and $> 0 \}$ 
can be at most countable. Since $\osc \notin Y$, $\varphi_- (\alpha) R_Y \pi$ implies $\alpha \in \varphi_+ (\pi)$.

Next consider the second case. Again assume without loss of generality that $0 \notin Y$. Let $\{ \varphi_- (\alpha) : \alpha < \cc \}$ be an almost 
disjoint family of subsets of $\omega$ without density (i.e., $d (\varphi_- (\alpha)) = \osc$). Again, $\varphi_+ (\pi) := \{ \alpha : d (\pi [\varphi_- (\alpha)]) $ is defined and $> 0 \}$ 
is at most countable. Since $0 \notin Y$, $\varphi_- (\alpha) R_Y \pi$ again implies $\alpha \in \varphi_+ (\pi)$, and the theorem is proved.
\end{proof}

Let $U\!R$ be the unreaping (or, unsplitting) relation on $[\omega]^\omega$, that is, $A\; U\!R \; B$ if $B \subseteq^* A$ or $B \cap A$ is finite.
This means that $A$ does not split $B$. $\dd ( [\omega]^\omega , [\omega]^\omega , U\!R)$ is the reaping number $\rr$ and $\bb ( [\omega]^\omega , [\omega]^\omega , U\!R)$
is the splitting number $\sss$.

\begin{theorem}   \label{density-reaping}
Assume $0\notin X$ and $1 \notin X$. Also assume
\begin{enumerate}
\item \underline{either} $0\in Y$ and $1\in Y$,
\item \underline{or} $\osc \notin X$ and there is a number $r \in ( 0,1]$ such that for all $x \in X$, both $r + x (1-r)$ and $x (1-r)$ belong to $Y$.
\end{enumerate}
Then $(D_X,  \Sym (\omega) , R_Y ) \leq_T ( [\omega]^\omega , [\omega]^\omega , U\!R)$.
\end{theorem}

Note that for $r = 1$, (2) is a special case of (1). By Theorem~\ref{density-nonmeager2} (2)  and the consistency of $\rr < \non (\M)$, the assumption $0,1 \notin X$
is necessary. If $\osc \in X$, then $0,1 \in Y$ is also necessary by Theorem~\ref{density-continuum}. However, if $\osc \notin X$, condition (2) can be relaxed:
e.g., we know that $\ddd_{(0,1) , \{  \osc\} } \leq \rr$, see Proposition~\ref{osc-all} and Corollary~\ref{dd-summary}.

\begin{corollary}   \label{density-reaping-cor}
Under the assumptions of Theorem \ref{density-reaping}, $\mathfrak{dd}_{X,Y} \leq \mathfrak{r}$ and $\mathfrak{s} \leq \mathfrak{dd}_{X,Y}^\perp$.
\end{corollary}

\begin{proof}[Proof of Theorem~\ref{density-reaping}]
We need to define $\varphi_-: D_X \to [\omega]^\omega$ and $\varphi_+ : [\omega]^\omega \to \Sym (\omega)$ such that 
for all $x \in D_X$ and $y \in [\omega]^\omega$, if $y \subseteq^* \varphi_- (x)$ or $y \cap \varphi_- (x)$ is finite, then $x R_Y \varphi_+ (y)$.

Assume first $0,1 \in Y$. 
Let $\varphi_- (x) = x$ for $x \in D_X$ and let $\varphi_+ (y)$ be such that $d (\varphi_+ (y) [y] ) = 1$ for $y \in [\omega]^\omega$. 
If $y \subseteq^* x$ then $d (\varphi_+ (y) [x] ) = 1$ and if $y \cap x$ is finite then $d (\varphi_+ (y) [x] ) = 0$. 
In either case $d(\varphi_+ (y) [x]) \neq d(x)$, and $x R_Y \varphi_+ (y)$ holds. 

Next assume $X \sub (0,1)$, and let $r$ be as required in (2). Again let $\varphi_- (x) = x$ for $x \in D_X$. Fix a set $z$ of density $r$.
For $y \in\omoms$, let $y' \sub y$ be a subset of density $0$ and let $\varphi_+ (y)$ be the unique permutation of $\omega$ that maps $y'$ to $z$ and
$\omega \sem y'$ to $\omega \sem z$, both in an order-preserving fashion. Note that $d (\omega \sem z) = 1-r$.

Now assume $y \sub^* x$ with $d(x) \in X$. Since $d (y') = 0$, the relative density of $x \sem y'$ in $\omega \sem y'$ is equal to $d(x)$ and, by definition of
$\varphi_+ (y)$, the relative density of $\varphi_+ (y) [x \sem y']$ in $\varphi_+ (y) [ \omega \sem y'] = \omega \sem z$ is $d(x)$ as well. Thus, as
$\varphi_+ (y) [y'] \sub^* \varphi_+ (y) [x]$, the density of $\varphi_+ (y) [x]$ is $r + d(x) (1-r) > d(x)$, which belongs to $Y$ by assumption. If, on the other
hand, $y \cap x$ is finite, the same reasoning shows the density of $\varphi_+ (y) [x]$ is $d(x) (1-r) < d(x)$. In either case, $x R_Y \varphi_+ (y)$ holds,
and the theorem is proved.
\end{proof}

As for the rearrangement numbers in~\cite[Theorem 5]{BBBHHL20}, for fixed $X \sub [0,1]$, the value of $\ddd_{X,Y}$ does not change as long as 
$\{ \osc \} \sub Y \sub \all$. This is based on the following lemma:

\begin{lemma}[{\cite[Lemma 4]{BBBHHL20}}]   \label{osc-lemma}
For any permutation $\pi \in \Sym (\omega)$ there exists a permutation $\sigma_\pi \in \Sym (\omega)$ such that
$\sigma_\pi [n] = n$ for infinitely many $n$ and $\sigma_\pi [n] = \pi [n]$ for infinitely many $n$.
\end{lemma}

\begin{proposition} \label{osc-all}
$\ddd_{X,\{ \osc \} } = \ddd_{X, \all}$ for all choices of $X \sub [0,1]$. 
\end{proposition}

\begin{proof}
Only $\ddd_{X, \{ \osc \} } \leq \ddd_{X,\all}$ needs proof. Let $\Pi$ be a witness for $\ddd_{X, \all}$ and show that $\Pi \cup \{  \sigma_\pi : \pi \in \Pi \}$ is
a witness for $\ddd_{X, \{ \osc \} }$. Let $x \in D_X$, and let $\pi \in \Pi$ be such that $d (\pi [x]) \neq d(x)$. If $d ( \pi [x]) = \osc$, we are done.
On the other hand, if $d (\pi [x]) \in [0,1]$, then, by the properties of $\sigma_\pi$, $d(\sigma_\pi [x]) = \osc$, and the proof is complete.
\end{proof}

Lemma~\ref{osc-lemma} is not helpful for proving the dual equality, and we do not know whether it holds. For rearrangement numbers it does, as shown by van der Vlugt  using an argument involving sequential composition of relations (see~\cite[Theorem 3.3.5]{vdV19}).

Let us summarize the results of this section about $\ddd_{X,Y}$ for the most interesting choices of $X$ and $Y$, namely, when either set is any of $\all$, $[0,1]$,
$(0,1)$, $\{ 0,1 \}$, or $\{ \osc \}$. This will subsume Corollary~\ref{density-nonmeager-cor2}.

\begin{corollary} \label{dd-summary}
\begin{enumerate}
\item If $\osc\in Y$, then $\ddd_{ [0,1], Y} = \ddd_{ \{ 0,1 \} , Y } = \non (\M)$.
\item If $\osc \notin Y$, then $\ddd_{ \all, Y} = \ddd_{ [0,1] ,Y } = \ddd_{ \{ 0,1 \} , Y}  = \cc$; also $\ddd_{ \{ \osc  \}, (0,1) } = \cc$.
\item $\cov (\N) \leq \ddd_{ (0,1) , \all} = \ddd_{ (0,1) , \{ \osc \} } \leq \min \{ \rr, \non (\M ) \}$.
\item $\max \{ \cov (\N), \cov (\M), \bb \} \leq \ddd_{(0,1), [0,1]} \leq \ddd_{(0,1), (0,1) } , \ddd_{(0,1), \{ 0,1 \} } \leq \rr$.
\item $\max \{ \bb, \cov (\M) \} \leq \ddd_{ \{ \osc \} , \all } = \ddd_{ \{ \osc \} , [0,1] } \leq \ddd_{ \{ \osc \} , \{ 0,1 \} } \leq \rr$.
\item $\max \{ \non (\M) , \cov (\M) \} \leq  \ddd_{ \all , \all } = \max \{ \non (\M) , \ddd_{\{ \osc \} , \all} \} \leq \max \{ \non (\M), \rr \}$.
\end{enumerate}
\end{corollary}

\begin{proof}
By Corollaries~\ref{density-nonmeager-cor2}, \ref{density-covnull-cor}, \ref{density-unbounding-cor}, \ref{density-covmeager-cor}, \ref{density-continuum-cor},
and~\ref{density-reaping-cor}.

Note that for $\ddd_{ (0,1) , \all} = \ddd_{ (0,1) , \{ \osc \} } $ in (3) we use Proposition~\ref{osc-all}. For (6) use 
$\ddd_{ \all , \all } = \max \{ \ddd_{ [0,1], \all } , \ddd_{\{ \osc \} , \all} \}$, (1), and (5).
\end{proof}

Table \ref{table} and Diagram \ref{diagram} below show the relationship between
the cardinals $\ddd_{X,Y}$ for the above five choices of $X$ and
$Y$ and some of the classical cardinal invariants of the continuum.

\begin{table}[htbp]
    \centering
\begin{tabular}{|c|c|c|c|c|}
\hline
   \diagbox{$X$}{$Y$}  & $(0,1)$ & $\{0,1\}$ (or $[0,1]$) & $\mathrm{all}$ & $\{\mathrm{osc}\}$ \\
\hline
 \multirow{2}{*}{$(0,1)$} & $\mathfrak{b}, \mathrm{cov}(\mathcal{M}), \mathrm{cov}(\mathcal{N})$ & $\mathfrak{b}, \mathrm{cov}(\mathcal{M}), \mathrm{cov}(\mathcal{N})$ & $\mathrm{cov}(\mathcal{N})$ & $\mathrm{cov}(\mathcal{N})$ \\
 & $\mathfrak{r}$ & $\mathfrak{r}$ & $\mathfrak{r}, \mathrm{non}(\mathcal{M})$ & $\mathfrak{r}, \mathrm{non}(\mathcal{M})$ \\
\hline
  $\{0,1\}$ & $\mathfrak{c}$ & $\mathfrak{c}$ & $\mathrm{non}(\mathcal{M})$ & $\mathrm{non}(\mathcal{M})$ \\
\hline
  $[0,1]$ & $\mathfrak{c}$ & $\mathfrak{c}$ & $\mathrm{non}(\mathcal{M})$ & $\mathrm{non}(\mathcal{M})$ \\
\hline
  \multirow{2}{*}{$\mathrm{all}$}  & \multirow{2}{*}{$\mathfrak{c}$} & \multirow{2}{*}{$\mathfrak{c}$} & $\mathrm{non}(\mathcal{M}),\mathrm{cov}(\mathcal{M})$ & \multirow{2}{*}{-} \\

& & & $\max\{\mathrm{non}(\mathcal{M}),\mathfrak{r}\}$  & \\
\hline
  \multirow{2}{*}{$\{\mathrm{osc}\}$}  & \multirow{2}{*}{$\mathfrak{c}$} & $\mathfrak{b}, \mathrm{cov}(\mathcal{M})$ & $\mathfrak{b}, \mathrm{cov}(\mathcal{M})$ & \multirow{2}{*}{-} \\
  &  & $\mathfrak{r}$ & $\mathfrak{r}$ & \\
\hline
\end{tabular}
\vspace{5pt}
\caption{Upper and lower bounds for the cardinals in 
Corollary~\ref{dd-summary}} \label{table} \small \justifying \vspace{-10pt} $X$ is in the left column and $Y$ in the top row; the upper row of cells containing two rows gives lower bounds, and the lower row, upper bounds. The columns for $Y=\{0,1\}$ and $Y=[0,1]$ coincide, so we merged them into one. This means the corresponding cardinals have the same bounds, but it is unclear whether they coincide. 
\end{table}

We shall prove two independence results (Theorems~\ref{density-lessthanb} and~\ref{density-largerthancofnull}) in the next section
showing that some of the inequalities in the corollary are consistently strict, but a number of questions remain open (see in particular
Questions~\ref{zeroone-covnull} and~\ref{zeroone-reaping} in Section 6). In particular, we do not know whether any of the $(X,Y)$-density numbers is distinct
from all classical cardinal invariants.

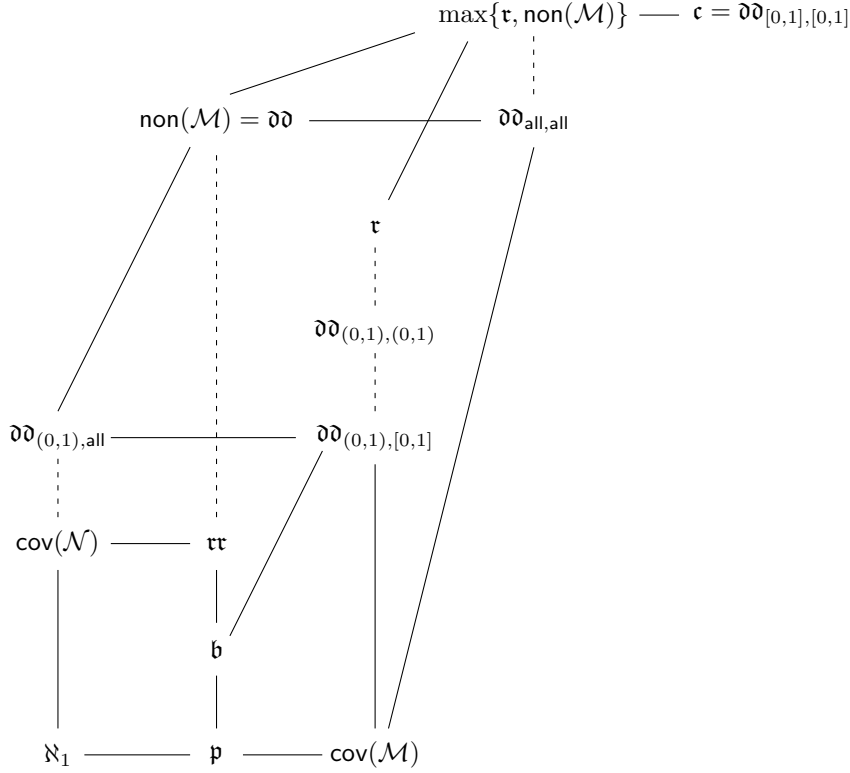
\begin{figure}[htbp]
\begin{picture}(300,300)(0,0)
\thinlines
\put(0,0){\makebox(40,20){$\aleph_1$}}
\put(60,0){\makebox(40,20){$\pp$}}
\put(60,40){\makebox(40,20){$\bb$}}
\put(60,80){\makebox(40,20){$\rrr$}}
\put(0,80){\makebox(40,20){$\cov (\N)$}}
\put(0,120){\makebox(40,20){$\ddd_{(0,1),\all}$}}
\put(120,0){\makebox(40,20){$\cov (\M)$}}
\put(120,120){\makebox(40,20){$\ddd_{(0,1),[0,1]}$}}
\put(120,160){\makebox(40,20){$\ddd_{(0,1),(0,1)}$}}
\put(120,200){\makebox(40,20){$\rr$}}
\put(60,240){\makebox(40,20){$\non (\M) = \ddd$}}
\put(180,240){\makebox(40,20){$\ddd_{\all,\all}$}}
\put(180,280){\makebox(40,20){$\max \{ \rr, \non (\M) \} $}}
\put(240,280){\makebox(100,20){$\cc = \ddd_{[0,1],[0,1]}$}}

\put(20,20){\line(0,1){60}}
\put(80,20){\line(0,1){20}}
\put(80,60){\line(0,1){20}}
\put(140,20){\line(0,1){100}}
\put(30,10){\line(1,0){40}}
\put(40,90){\line(1,0){30}}
\put(90,10){\line(1,0){30}}
\put(40,130){\line(1,0){70}}
\put(115,250){\line(1,0){65}}
\put(240,290){\line(1,0){15}}
\put(85,55){\line(1,2){35}}
\put(145,20){\line(1,4){55}}
\put(145,220){\line(1,2){30}}
\put(20,140){\line(1,2){50}}
\put(85,260){\line(3,1){70}}

\multiput(20,100)(0,5){5}{\line(0,1){2}}
\multiput(140,140)(0,5){5}{\line(0,1){2}}
\multiput(140,180)(0,5){5}{\line(0,1){2}}
\multiput(200,260)(0,5){5}{\line(0,1){2}}
\multiput(80,100)(0,5){28}{\line(0,1){2}}
\end{picture}
\vspace{-10pt}
\caption{Some cardinals}
\small \justifying
For simplicity, we consider here only the cases $X,Y = \all, (0,1), [0,1]$: cardinals grow as one moves right and up; straight lines
signify that strict inequality is consistent, dotted lines mean that we do
not know whether equality holds or strict inequality is consistent. Note
that $\ddd = \ddd_{[0,1], \all}$ and that $\cc$ is also equal to each
of $\ddd_{[0,1],(0,1)}$, $\ddd_{\all,[0,1]}$, and $\ddd_{\all,(0,1)}$.
\label{diagram}
\end{figure}



\section{Variations on the density number: consistency results}


Our first consistency result says that the unbounding number $\bb$ is not a lower bound of $\ddd_{ (0,1) , \all}$;
it is based on the fact that $\sigma$-centered forcing does not increase the latter cardinal.

\begin{theorem}  \label{density-lessthanb}
    ${\mathfrak{dd}}_{X,\all}  < \mathfrak{p}$ is consistent for all $X \sub (0,1)$. In particular, 
    $\mathfrak{dd}_{(0,1), \all} < \mathfrak{b}$ and $\mathfrak{dd}_{(0,1), \all} < \mathfrak{dd}$ are consistent.
\end{theorem}

\begin{proof}
    Let $(I_n : n \in \omega)$ be an interval partition of $\omega$ such that $|I_{n+1}| \geq n \cdot \sum_{i \leq n} | I_i |$ for all $n$. Let $k\in\omega$ and $j$ such that $0 < j < k-1$. Say that $A \subseteq \omega$ is a {\em $(j,k)$-set} if $${ | A \cap I_n | \over |I_n|} \in \left[ {j\over k}, {j+1 \over k } \right]$$ for almost all $n \in \omega$. It is easy to see that if $A$ is a $(j,k)$-set then $\underline{d} (A) \leq {j+1 \over k}$ and $\overline{d} (A) \geq {j \over k}$. In particular, if $A$ has asymptotic density then $d(A) \in \left[ {j \over k }, {j + 1 \over k } \right]$. A permutation $\pi$ {\em disrupts} the $(j,k)$-set $A$ if $\pi$ fixes all $I_n$ and for infinitely many $n$, $\pi [ A \cap I_n]$ is an initial segment of $I_n$. Note that $\pi [A]$ is still a $(j,k)$-set. 

\begin{lemma}
    If $A$ is a $(j,k)$-set and $\pi$ disrupts $A$ then $\overline{d} ( \pi [A] ) = 1$. In particular, $\pi[A]$ does not have asymptotic density.
\end{lemma}

\begin{proof}
    Let $\varepsilon > 0$. Let $n > {k \over \varepsilon}$ such that $\pi [A \cap I_{n+1}]$ is an initial segment of $I_{n+1}$. Let $\ell = \max ( \pi [ A \cap I_n] ) + 1$. Then $${| \pi [A] \cap \ell | \over \ell} \geq { {j\over k} |I_{n+1}| \over \sum_{i\leq n} |I_i| + {j \over k} |I_{n+1}| } \geq { {j \over k} | I_{n+1} | \over \left({1 \over n} + {j \over k} \right) |I_{n+1}| } \geq {j \over \varepsilon + j} = {j (j - \varepsilon) \over j^2 - \varepsilon^2} \geq 1 - {\varepsilon \over j} $$ as required.
\end{proof}

\begin{lemma}   \label{density-lessthanb-2}
    Let ${\mathbb{P}}$ be a $\sigma$-centered forcing notion. Let $\dot A$ be a ${\mathbb{P}}$-name for a $(j,k)$-set. Then there are $(j,k)$-sets $(A_m : m \in \omega)$ such that if $\pi$ disrupts all $A_m$, then $\pi$ is forced to disrupt $\dot A$.
\end{lemma}

\begin{proof}
    Assume ${\mathbb{P}} = \bigcup_m P_m$ where all $P_m$ are centered. For each $m$ and $n$ let $A_m \cap I_n$ be such that no condition in $P_m$ forces that $\dot A \cap I_n \neq A_m \cap I_n$. Since there are only finitely many possibilities for this intersection the centeredness of $P_m$ implies that such a $A_m \cap I_n$ can indeed be found. Now assume that $\pi$ disrupts all $A_m$. Let $n_0 \in \omega$ and $p \in {\mathbb{P}}$. There is $m$ such that $p \in P_m$. Find $n \geq n_0$ such that $\pi [ A_m \cap I_n]$ is an initial segment of $I_n$. There is $q \leq p$ such that $q \forces \dot A \cap I_n = A_m \cap I_n$. So $q$ forces that $\pi [\dot A \cap I_n]$ is an initial segment of $I_n$.
\end{proof}

The following is a standard argument for finite support iterations of ccc forcing. We include the proof for the sake of completeness.

\begin{lemma}  \label{density-lessthanb-3}
    Let $\delta$ be a limit ordinal, and let $(\mathbb{P}_\alpha : \alpha \leq \delta)$ be a finite support iteration of ccc forcing such that for any $\alpha < \delta$ and any $\mathbb{P}_\alpha$-name $\dot A$ for a $(j,k)$-set there are $(j,k)$-sets $(A_m : m \in \omega)$ such that if $\pi$ disrupts all $A_m$, then $\pi$ is forced to disrupt $\dot A$. Then for any $\mathbb{P}_\delta$-name $\dot A$ for a $(j,k)$-set there are $(j,k)$-sets $(A_m : m \in \omega)$ such that if $\pi$ disrupts all $A_m$, then $\pi$ is forced to disrupt $\dot A$.
\end{lemma}

\begin{proof}
If $cf (\delta) > \omega$ there is nothing to show. So assume without loss of generality that $\delta = \omega$, and let $\dot A$ be a $\PP_\omega$-name for a 
$(j,k)$-set. Fix $m \in \omega$. In $V[G_m]$, where $G_m$ is $\PP_m$-generic over $V$, we may find a decreasing sequence of conditions $(p^\ell_m : \ell \in \omega )$
in the remainder forcing $\PP_{[m,\omega)}$ deciding $\dot A$, i.e., there is a $(j,k)$-set $A_m$ such that $p^\ell_m \forces \dot A \cap I_\ell = A_m \cap I_\ell$ for all
$\ell \in \omega$. (We may assume $p^0_m$ decides from which $I_n$ onwards $\dot A$ satisfies the condition of being a $(j,k)$-set, and then $A_m$ will automatically
be a $(j,k)$-set as well.) Back in the ground model $V$, let $(\dot p^\ell_m : \ell \in \omega)$ and $\dot A_m$ be $\PP_m$-names for $(p^\ell_m : \ell \in \omega )$
and $A_m$, respectively.

By assumption, for each $m$ there are $(j,k)$-sets $(A^n_m : n \in \omega )$ such that if $\pi$ disrupts all $A^n_m$, $n\in\omega$, then $\pi$
is forced to disrupt $\dot A_m$. We show that the family $(A^n_m : n,m \in \omega)$ is as required for $\dot A$. Assume $\pi$ disrupts all $A^n_m$, $n,m \in\omega$.
Let $p \in \PP_\omega$ and $\ell_0 \in \omega$ be arbitrary. We need to find $q \leq p$ and $\ell \geq \ell_0$ such that $\pi [ \dot A \cap I_\ell]$ is forced to be an initial
segment of $I_\ell$. Find $m$ such that $p \in \PP_m$. Find $p' \leq p$ in $\PP_m$ and $\ell \geq \ell_0$ such that $p' \forces `` \pi [ \dot A_m \cap I_\ell] $ is an
initial segment of $I_\ell$". We may assume $p'$ decides $\dot p^\ell_m$ as well, $p' \forces \dot p^\ell_m = p^\ell_m$. Let $q = p' \cup p_m^\ell$.
Then $q \forces \dot A \cap I_\ell = \dot A_m \cap I_\ell$ and thus $q$ also forces that $\pi [ \dot A \cap I_\ell]$ is  an initial
segment of $I_\ell$, as required.
\end{proof}

We are ready to complete the proof of the theorem. Assume CH, let $\kappa > \aleph_1$ be a regular cardinal, and let $({\mathbb{P}}_\alpha, \dot {\mathbb{Q}}_\alpha : \alpha < \kappa)$ 
be a finite support iteration of $\sigma$-centered partial orders forcing $\mathfrak{p} = \kappa$. It suffices to show that in the generic extension $V[G_\kappa]$,
$\mathfrak{dd}_{ (0,1), \all} = \aleph_1$, and that this is witnessed by the ground model permutations. Let $r \in (0,1)$ in $V[G_\kappa]$.
Let $\dot r$ be a $\PP_\kappa$-name for this real. We can find a condition $p \in {\mathbb{P}}_\kappa$ and numbers $k$ and $j$ with $0 < j < k-1$ 
such that $p$ forces $\dot r \in \left( {j\over k }, {j+1 \over k} \right)$. By strengthening $p$, if necessary, we may assume that for some $m\geq 3$, 
$$ p \forces \dot r \in \left[ {mj + 1 \over mk} , {mj + m -1 \over mk} \right]$$ 
Work below $p$. Let $\dot A$ be a $\mathbb{P}_\kappa$-name for a set of density $\dot r$. We first check that $p$ forces that $\dot A$ is a $(j,k)$-set.

To see this work in the generic extension with $p$ belonging to the generic filter. Suppose this fails and assume (without loss of generality) that there are infinitely many $n$ with ${|A \cap I_n| \over |I_n|} > {j+1 \over k}$. Let $n$ be such that ${|A \cap I_{n+1}| \over |I_{n+1}|} > {j+1 \over k}$. Then $${| A \cap ( \max (I_{n+1}) + 1 ) | \over \max (I_{n+1}) + 1} \geq { | A \cap I_{n+1}| \over \sum_{i \leq n} |I_i| + |i_{n+1}| } \geq { |A \cap I_{n+1} | \over \left( 1 + {1 \over n} \right) |I_{n+1}| } > {j+1 \over k} \cdot {1 \over {1 + {1\over n}}} $$
which is larger than ${mj + m - 1 \over mk}$ for large enough $n$, a contradiction.

By the two previous lemmas, there are $(j,k)$-sets $(A_m : m \in \omega)$ in the ground model such that if $\pi$ disrupts all $A_m$ then $p$ forces that $\pi$ disrupts $\dot A$. Let $\pi$ be a permutation satisfying this assumption. Thus $p$ forces that $\pi [\dot A]$ does not have asymptotic density (by the above it forces $\underline{d} (\pi [\dot A]) \leq {j +1 \over k}$ and $\overline{d} (\pi [\dot A] ) = 1$). This completes the argument showing that the ground model permutations witness $\mathfrak{dd}_{ (0,1), \all}  = \aleph_1$.
\end{proof}

\begin{theorem}   \label{density-largerthancofnull}
    $\mathfrak{dd}_{ (0,1) , [0,1] } > \cof (\mathcal N)$ is consistent.
\end{theorem}

Note that this establishes in particular that $\mathfrak{dd}_{ (0,1) , [0,1] } > \max \{ \bb, \cov (\N), \cov (\M) \}$ is consistent (this max is a lower bound of $\mathfrak{dd}_{ (0,1) , [0,1] }$
by item (4) of Corollary~\ref{dd-summary}). 


\begin{proof}
    This holds in the Silver model, that is, the model obtained by the $\omega_2$-stage countable support iteration of Silver forcing $\mathbb{S}$. 
    $\cof (\mathcal N) = \aleph_1$ is well-known (this is so because
    the countable support product of Silver forcing has the Sacks property~\cite[Lemma 24.2 and Related Result 132]{Ha17}, and the Sacks property implies that bases of the
    null ideal are preserved~\cite[Lemma 6.3.39]{BJ95}). We show $\mathfrak{dd}_{ \{ {1\over 2} \} , [0,1] } = \aleph_2 = \cc$. 

    Let $(I_n : n\in\omega)$ be an interval partition of $\omega$ such that $|I_n| = 2^n$. Let $s_n \in 2^{I_n}$ be an alternating sequence of zeros and ones, i.e., $$s_n ( \min (I_n)) = 0, s_n (\min (I_n) + 1) = 1 , s_n (\min (I_n) + 2) = 0, ...$$ Let $\bar s_n = 1 - s_n$.

    Recall that Silver forcing $\mathbb{S}$ consists of functions $f : \omega \to 2 $ with infinite codomain. The order is given by $g \leq f $ if $g \supseteq f$. If $G$ is $\mathbb{S}$-generic over $V$, $x = \bigcup \{ f \in \mathbb{S} : f \in G \}$ is the generic Silver real. Fix $n \in \omega$. If $| ( x \restriction n)^{-1} ( \{ 1 \} ) |$ is even let $a_n = s_n$, if it is odd let $a_n = \bar s_n$. Put $a = \bigcup_n a_n \in 2^\omega$ and think of $a$ as a subset of $\omega$. By construction $d(a) = {1\over 2}$. Let $\dot a$ be an $\mathbb{S}$-name for $a$.

    Let $f \in \mathbb{S}$ and let $\pi$ be a permutation in $V$. Suppose that for some $r > 0$, $f \forces \overline{d} ( \pi [ \dot a ]) \geq {1 \over 2} + r$. We show that $f \forces \underline{d} ( \pi [\dot a]) \leq {1\over 2} - r$. In particular, the ground model permutations are not a witness for $\mathfrak{dd}_{ \{ {1\over 2} \} , [0,1] }$, and  $\mathfrak{dd}_{ \{ {1\over 2} \} , [0,1] } = \aleph_2$ in the Silver model follows.

    Let $g \leq f$ and $\varepsilon > 0$ be arbitrary. Let $n_0 = \min ( \omega \setminus \mathrm{dom} (g))$. Let $k_0$ be so large that $\pi [ \bigcup_{i \leq n_0} I_i] \subseteq k_0$ and $\pi^{-1} [ \bigcup_{i \leq n_0} I_i] \subseteq k_0$ and such that for some $g' \leq g $, $$g' \forces { | \pi [ \dot a] \cap k_0 | \over k_0} \geq {1\over 2} + r - \varepsilon$$ Let $b_0 = \pi^{-1} [k_0] \cap (\omega \setminus k_0)$ and $b_1 = \pi^{-1} [ \omega \setminus k_0] \cap k_0$. Clearly $|b_0| = |b_1|$ and $\pi [b_0] \subseteq k_0 \setminus \bigcup_{i \leq n_0} I_i$ and $b_1 \subset k_0 \setminus \bigcup_{i \leq n_0} I_i$. Let $n_1$ be such that $b_0 , \pi[b_1] \subseteq \bigcup_{i \leq n_1} I_i$. We may assume $n_1 \subseteq \mathrm{dom} (g')$. This means that $g'$ decides $\bigcup_{i \leq n_1} \dot a_i = \dot a \restriction \bigcup_{i \leq n_1} I_i$. Now let $g''$ be such that $g'' \leq g$, $\mathrm{dom} (g'') = \mathrm{dom} (g')$, $g'' (n_0) = 1 - g' (n_0)$, and $g'' (k) = g' (k) $ for $k \in \mathrm{dom} (g') \setminus \{ n_0 \}$. Letting $c_0$ and $c_1$ be the values forced to $\dot a \restriction b_0$ and $\dot a \restriction b_1$ by $g'$, respectively, we see immediately that $g'' \forces ``\dot a \restriction b_0 = 1 - c_0$ and $\dot a \restriction b_1 = 1 - c_1$". Since the trivial condition forces $\pi [\dot a] \cap k_0 = [ ( \dot a \cap k_0) \cup \dot a \restriction b_0] \setminus \dot a \restriction b_1$, we see that $$g'' \forces { | \pi [\dot a] \cap k_0 | \over k_0} \leq {1\over 2} - r + \varepsilon$$ as required.
\end{proof}

Presumably a modification of this argument shows that $\mathfrak{dd}_{ \{ r \} , [0,1] } = \cc$ holds for any $r$ and not just ${1 \over 2}$ in the Silver model. 

We can simultaneously separate four of the density numbers:

\begin{theorem}    \label{four-density}
It is consistent that $\aleph_1 = \ddd_{ (0,1) , \all} < \ddd = \non 
(\M) < \cov (\M) = \ddd_{ (0,1) , [0,1] } = \rr < \ddd_{ [0,1] , [0,1] } = \cc$.
\end{theorem}

\begin{proof}[Proof sketch] This is a standard argument, which we
will only sketch. Assume for simplicity GCH. Let $\aleph_1 < \kappa <
\lambda < \mu$ be cardinals with $\kappa$ and $\lambda$ being regular and
$\mu$ of uncountable cofinality. Let $f : \lambda \to \kappa$ be
a function such that $f^{-1} ( \{ \alpha \} ) $ is cofinal in
$\lambda$ for each $\alpha < \kappa$. First add $\mu$ Cohen reals to
force $\cc = \mu$.

Then perform a {\em matrix iteration}, that is, a two-dimensional
system of finite support iterations $\langle \PP_{\alpha,\gamma}, \dot \QQ_{\alpha,\gamma} : \gamma < \lambda \rangle$, $\alpha < \kappa$,
such that 
\begin{enumerate}
    \item $\PP_{\alpha,0}$ is the forcing adding $\alpha$ Cohen
    reals $(c_\beta : \beta < \alpha)$,
    \item $\PP_{\alpha , \gamma + 1} $ is the two-step iteration 
    $\PP_{\alpha, \gamma} \star \dot \QQ_{\alpha,\gamma} $, and if
    $\alpha < f(\gamma) $, then $\dot \QQ_{\alpha,\gamma}$ is a
    $\PP_{\alpha,\gamma}$-name for the trivial forcing, while if
    $\alpha \geq f(\gamma)$, then $\dot \QQ_{\alpha,\gamma}$ is a
    $\PP_{\alpha,\gamma}$-name for Mathias forcing $\MM_{\dot \U_\gamma}$
    where $\dot \U_\gamma$ is a $\PP_{f(\gamma), \gamma}$-name 
    for an ultrafilter in $V_{f(\gamma),\gamma}$,
    \item if $\gamma$ is a limit ordinal, then $\PP_{\alpha,\gamma}$
    is the direct limit of the $\PP_{\alpha,\delta}$, $\delta<\gamma$,
    \item $\PP_{\alpha,\gamma} \embed \PP_{\beta,\delta}$ for
    $\alpha \leq \beta$ and $\gamma\leq\delta$.
\end{enumerate}
As usual $V_{\alpha,\gamma}$ denotes the intermediate extension via $\PP_{\alpha,\gamma}$.
It is clear that this construction can be carried out. For
item (4) note that $\PP_{\alpha,\gamma} \embed \PP_{\alpha,\delta}$
is trivial because we are dealing with a standard iteration
while $\PP_{\alpha,\delta} \embed \PP_{\beta,\delta}$ is proved by
induction on $\delta < \lambda$. The basic step is obvious by (1), and
so is the successor step by (2). For the limit step use e.g.~\cite[Lemma 10]{BF11}.

We need to see that the $\PP_{\kappa,\lambda}$-generic extension 
$V_{\kappa,\lambda}$ is the required model.
Clearly $\ddd_{[0,1],[0,1]} = \cc = \mu$ is preserved
(see Corollary~\ref{dd-summary} (2)), and $\ddd_{(0,1), \all}=\aleph_1$
follows from the fact that all iterands are $\sigma$-centered and 
the techniques of the proof of Theorem~\ref{density-lessthanb}
(in particular Lemmas~\ref{density-lessthanb-2} and~\ref{density-lessthanb-3}).

Next, since $\langle \PP_{\kappa,\gamma} : \gamma < \lambda \rangle$
is a finite support iteration of length $\lambda$, $\cov (\M) 
\geq \lambda$ follows. On the other hand, the Cohen reals 
$\{ c_\alpha : \alpha < \kappa \}$ form a non-meager set of size $\kappa$
in the final extension so that $\non (\M) \leq \kappa$ holds.

To see the latter, it suffices to prove by induction on $\gamma \leq\lambda$
that \[\PP_{\alpha,\gamma} \times \CC_\alpha \embed \PP_{\alpha +1,
\gamma} \;\;\; (\star)\] for all $\alpha < \kappa$, where $\CC_\alpha$ adds the
$\alpha$-th Cohen real $c_\alpha$ in the initial step, that is,
$\PP_{\alpha + 1, 0} = \PP_{\alpha, 0} \times \CC_\alpha$. In particular,
for $\gamma = 0$, $(\star)$ is obvious. Next assume $\gamma =
\delta + 1$ is successor. If $\alpha < f (\delta)$, then 
$\PP_{\alpha,\gamma} = \PP_{\alpha,\delta}$ and $(\star)$ is 
immediate by induction hypothesis. If $\alpha \geq f(\delta)$,
then $\PP_{\beta, \gamma} = \PP_{\beta,\delta} \star \MM_{\dot
\U_\delta}$ for $\beta \in \{\alpha, \alpha + 1 \}$. Therefore
by induction hypothesis and the product lemma, we obtain 
\[ \PP_{\alpha,\gamma} \times \CC_\alpha = (\PP_{\alpha,\delta} \star \MM_{\dot
\U_\delta} )\times \CC_\alpha = (\PP_{\alpha,\delta} \times
\CC_\alpha) \star \MM_{\dot \U_\delta} \embed 
\PP_{\alpha + 1,\delta} \star \MM_{\dot \U_\delta} = 
\PP_{\alpha + 1,\gamma} \]
as required. Finally, if $\gamma$ is a limit ordinal,  a slight
modification of the argument of~\cite[Lemma 10]{BF11} works:
let $\{ (p_n,c_n) : n \in \omega \}$ be a maximal antichain in
$\PP_{\alpha, \gamma} \times \CC_\alpha$, and let $p$ be a condition
in $\PP_{\alpha + 1, \gamma}$. There is $\delta < \gamma$ such that
$p \in \PP_{\alpha + 1, \delta}$. The projection $\{ (p_n \re \delta,
c_n) : n \in \omega \}$ of the antichain to $\PP_{\alpha, \delta} \times
\CC_\alpha$ is a predense set. Therefore, by induction hypothesis,
there is $n$ such that $p$ and $(p_n \re \delta, c_n)$ are compatible
with common extension $q \in \PP_{\alpha + 1, \delta}$. It is
now obvious that $q$ is also compatible with $(p_n,c_n)$ in 
$\PP_{\alpha + 1,\gamma}$, and $(\star)$ follows.

For each $\gamma < \lambda$  let $m_\gamma$ be the Mathias-generic
added by $\MM_{\U_\gamma}$ over $V_{f(\gamma) , \gamma}$, and
note that $\{ m_\gamma : \gamma < \lambda \}$ is an unreaped family so 
that $\rr \leq \lambda$ follows. $\cov (\M) = \dd_{ (0,1), [0,1] } =
\rr = \lambda$ is now immediate by Corollary~\ref{dd-summary} (4).
Finally, if $A \sub [\omega]^\omega$ has size less than $\kappa$,
then there are $\gamma < \lambda$ and $\alpha < \kappa$ such that
$A \in V_{\alpha, \gamma}$. Choosing $\delta \geq \gamma$ such that
$f(\delta ) \geq \alpha$ we see that the Mathis generic $m_\delta$
is not split by any member of $A$. Thus $\sss \geq \kappa.$
Since $\sss \leq \non (\M)$ holds in ZFC, $\sss = \non (\M) =
\kappa$ follows, and the proof is complete.
\end{proof}



\section{Questions and final remarks}


As mentioned after Theorem~\ref{density-covnull}, one of the questions we could not answer is the following:

\begin{question}   \label{osc-covnull}
Is $\ddd_{ \{ \osc \} , \all} \geq \cov (\N)$?
\end{question}

This is really a problem about the random model: if random forcing adds a set without density such that for all ground model
permutations, the permuted set still has no density, then the answer is yes. If random forcing does not add such a set,
$\ddd_{ \{ \osc \} , \all} < \cov (\N)$ is consistent (and holds in the random model).

We know by item (3) of Corollary~\ref{dd-summary} that $\cov (\N) \leq \ddd_{ (0,1), \all } \leq \min \{ \rr, \non (\M) \}$ and by
Theorem~\ref{density-lessthanb} the second inequality can be consistently strict. However, we do not know whether 
one can separate $\ddd_{ (0,1), \all } $ from $\cov (\mathcal N)$.

\begin{question}   \label{zeroone-covnull}
Is $\cov (\mathcal N) < \ddd_{ (0,1), \all}$ consistent? Or are the cardinals equal?
\end{question}

Similarly, by (4) of Corollary~\ref{dd-summary}, we know $\max \{ \bb, \cov (\N), \cov (\M) \} \leq \ddd_{ (0,1), [0,1]} \leq \rr$ and by Theorem~\ref{density-largerthancofnull}
the first inequality is consistently strict, but we do not know the answer to the following:

\begin{question} \label{zeroone-reaping}
Is $\ddd_{ (0,1) , [0,1] } < \rr$ consistent? Or are the cardinals equal?
\end{question}

Note that by Theorem~\ref{density-nonmeager2}, $\ddd_{ \{ 0 \} , \all } = \ddd_{ \{ 1 \} , \all } = \non (\M)$. Also, it is easy to see that
$\ddd_{ \{ r \} , \all } = \ddd_{ \{ 1-r \} , \all } $ for any $r \in (0,1)$ (and similarly with $\all$ replaced by other natural choices like $[0,1]$, $ \{ 0,1 \}$, or $(0,1)$),
though these cardinals may be strictly smaller than $\non (\M)$. We do not know:

\begin{conjecture}
$\ddd_{ \{ r \} , \all }  = \ddd_{ \{ { 1\over 2}  \} , \all } $ for all $r \in (0,1)$ (and similarly with $\all$ replaced by $[0,1]$, $ \{ 0,1 \}$, or $(0,1)$).
\end{conjecture}

A more sweeping conjecture would be:

\begin{conjecture}
Assume $X , X' \sub \all$ are such that $X \sem (0,1) = X' \sem (0,1)$, and let $Y$ be arbitrary. Then   $\ddd_{X,Y}  = \ddd_{ X' , Y} $.
\end{conjecture}

This would mean $\ddd_{X,Y}$ is completely independent of the intersection $X \cap (0,1)$. It probably depends on $Y \cap (0,1)$, at least on
its size, though we have not pursued  this (see Theorem~\ref{cof} for a related result).

A  related  question is how many of the density numbers can be (consistently) simultaneously distinct. We do have models with four values (Theorem~\ref{four-density}).

\begin{question}
Can five or more density numbers of the form $\ddd_{X,Y}$ be simultaneously distinct? Can infinitely many be simultaneously distinct?
\end{question}



As for the density number, there are several natural variants of the rearrangement number, and some of them have been considered in the literature~\cite{BBBHHL20}.
Let us introduce a general framework similar to Definition~\ref{general-density-number}. We use again the symbol $\osc$, and write $\sum_n a_n = \osc$ if the series
$\sum_n a_n$ diverges by oscillation. Let $\all = \RR \cup \{ + \infty, - \infty, \osc \}$.

\begin{definition}   \label{general-rearrangement-number}
Assume $X, Y \sub \all$ are such that $X \neq Y$ and for all $x \in X$ there is $y \in Y$ with $y \neq x$. The {\em $(X,Y)$-rearrangement number} $\rrr_{X,Y}$ is the
smallest cardinality of a family $\Pi \sub \Sym (\omega)$ such that for every p.c.c. series $\sum_n a_n$ with $\sum_n a_n \in X$ there is $\pi \in \Pi$ such that
$\sum_n a_{\pi (n) } \in Y$ and $\sum_n a_{\pi (n)} \neq \sum_n a_n$.
\end{definition}

Again, $X' \sub X$ and $Y'  \supseteq Y$ obviously imply $\rrr_{X' , Y'} \leq \rrr_{X,Y}$. Also $\rrr = \rrr_{\RR,\all}$ and $\rrr ' =\rrr_{\RR \cup \{ \pm \infty \}, \all } ( = \non (\M))$, where $\all = \RR \cup\{\pm \infty\}\cup \{\osc\}$.
In~\cite{BBBHHL20}, the two cardinals $\rrr_i = \rrr_{ \RR , \{ \pm \infty \} }$ and $\rrr_f = \rrr_{\RR,\RR}$ were thoroughly investigated, and it was proved that 
both are above the dominating number $\dd$ and consistently strictly below the continuum $\cc$~\cite[Theorems 12, 29, and 38]{BBBHHL20}. With respect to the analogy of the table in Section 2,
$\rrr$ would correspond to $\ddd_{ (0,1), \all}$, $\rrr_i$ to $\ddd_{ (0,1), \{ 0,1 \} } $, and $\rrr_f$ to $\ddd_{(0,1), (0,1)}$, respectively. However,
$\rrr$ is above $\bb$ while $\ddd_{ (0,1), \all}$ is consistently below $\bb$ (Theorem~\ref{density-lessthanb}); also, as mentioned above, $\rrr_i$ and $\rrr_f$ are above $\dd$ while
$\ddd_{ (0,1), \{ 0,1 \} } $ and $\ddd_{(0,1), (0,1)}$ are below $\rr$ (Corollary~\ref{dd-summary} (4)) with $\rr < \dd$ being consistent (this holds in the Miller 
and the Blass-Shelah models, see~\cite[11.9]{Bl10} and~\cite{BS87}).
Why does the analogy break down?

The rough answer is that while we can replace a series converging to a value $r$ by a series with the same positive and negative terms in the same order, and
thus converging to the same value $r$, but containing long intervals of zeroes (this is called ``padding with zeroes" in~\cite{BBBHHL20}), there is no
corresponding operation for infinite-coinfinite sets $A$. Padding with zeroes should correspond to introducing elements belonging neither to $A$ nor to its complement.
This suggests we should consider the relative density of $A$ in some larger set $B$, with $\omega \sem B$ playing the role of the set of padded zeroes.

Let $A \sub B \sub \omega$. Define the {\em lower relative density of $A$ in $B$}
\[ \underline{d}_B (A) = \lim\inf_{n \to \infty} { |A \cap n| \over |B \cap n | } = \lim\inf_{n \to \infty} { d_n (A) \over d_n (B) } \]
and the {\em upper relative density of $A$ in $B$}
\[ \overline{d}_B (A) = \lim\sup_{n \to \infty} { |A \cap n| \over |B \cap n | } = \lim\sup_{n \to \infty} { d_n (A) \over d_n (B) } \]
If $\underline{d}_B (A) = \overline{d}_B (A)$, the common value $d_B (A)$ is the {\em relative density of $A$ in $B$}.
As for density, the interesting case is when both $A$ and $B \sem A$ are infinite. As in Section 3, say $d_B (A) = \osc$ if
$\underline{d}_B (A) < \overline{d}_B (A)$, and let $\all = [0,1] \cup \{ \osc \}$.

\begin{definition}
Assume $X,Y \sub \all$ are such that $X \neq \emptyset$ and for all $x \in X$ there is $y \in Y$ with $y \neq x$. The {\em $(X,Y)$-relative density number} 
$\ddd_{X,Y}^{\rel}$ is the smallest cardinality of a family $\Pi \sub \Sym (\omega)$ such that for every $A \sub B \sub \omega$ with $A$ and $B \sem A$
both infinite and with $d_B (A) \in X$ there is $\pi \in \Pi$ such that $d_{\pi [B]} ( \pi [A] ) \in Y$ and $d_{\pi [B]} ( \pi [A] ) \neq d_B (A)$.
\end{definition}

Let us provide the Tukey framework for these cardinals: for $X,Y$ as in the definition, consider triples $(D_X^\rel , \Sym (\omega) , R_Y^\rel)$
where $D_X^\rel$ is the collection of pairs $(A,B)$ such that $A \sub B \sub \omega$ with both $A$ and $B \sem A$ infinite and $d_B (A) \in X$,
and the relation $R_Y^\rel$ is given by $(A,B) R_Y^\rel \pi$ if $d_{\pi [B]} ( \pi [A] ) \in Y$ and $d_{\pi [B]} ( \pi [A] ) \neq d_B (A)$,
for $(A,B) \in D_X^\rel$ and $\pi \in \Sym (\omega)$. Then $\ddd_{X,Y}^\rel = \dd ( D_X^\rel , \Sym (\omega) , R_Y^\rel)$, and we let 
$\ddd_{X,Y}^{\rel, \perp} = \bb( D_X^\rel , \Sym (\omega) , R_Y^\rel)$.

\begin{proposition} \label{reldensity-nonmeager}
\begin{enumerate}
\item $( D_X , \Sym (\omega) , R_Y) \leq_T ( D_X^\rel , \Sym (\omega) , R_Y^\rel)$ for any $X,Y$.
\item $( D_X^\rel , \Sym (\omega) , R_Y^\rel) \leq_T (\M ,\twoom, \not\ni)$ if $\osc \in Y$ and $\osc \notin X$.
\end{enumerate}
\end{proposition}

In particular, $\ddd_{X,Y} \leq \ddd_{X,Y}^{\rel}$ always holds.

\begin{corollary}
 Suppose $X,Y$ are such that $\osc\not\in X$, $X\cap\{0,1\}\neq\emptyset$ and $\osc\in Y$. 
 Then $\ddd_{X,Y}^\rel=\non(\M)$.
    
    \end{corollary}

\begin{proof}[Proof of Proposition~\ref{reldensity-nonmeager}]
For (1), take $\varphi_+$ to be the identity function, $\varphi_+ (\pi)=\pi$, and
$\varphi_- ( A ) = ( A, \omega )$, for $A \in D_X$.

The proof of (2) is analogous to the first part of 
the proof of Theorem~\ref{density-nonmeager}. 
\end{proof}

The next two theorems (and their corollaries) should be seen as analogues of the results $\bb \leq \rrr$ and $\dd \leq \rrr_i, \rrr_f$
about rearrangement numbers~\cite[Theorems 11 and 12]{BBBHHL20}.

\begin{theorem} \label{reldensity-unbounding}
$(\omom, \omom , \not\geq^*) \leq_T ( D_X^\rel , \Sym (\omega) , R_Y^\rel)$ for any $X,Y$.
\end{theorem}

\begin{corollary}
    $\bb\leq\ddd_{X,Y}^{\rel}$ for any $X,Y$. 
\end{corollary}

\begin{proof}[Proof of Theorem~\ref{reldensity-unbounding}]

From ~\cite{BBBHHL20}, recall the definition of the cardinal $\jj$ and the proof that $\bb=\jj$. There, the argument for $\bb\leq\jj$ can be easily modified to give a reduction $(\omom, \omom , \not\geq^*) \leq_T( \omoms, \Sym (\omega), \centernot P )$, where $A P\pi$ means that $\pi$ does not change the relative order of members of $A$ except possibly for finitely many elements (in this case, we say that \emph{$\pi$ preserves $A$}).
Indeed, for each $\pi\in\Sym (\omega)$, let $\psi_+(\pi)= f_\pi\in\omom$ be
such that $n<f_\pi(n)$ and $\pi(x)<\pi(y)$ for all $x\leq n<f_\pi(n)\leq y$.
Also, for each strictly increasing $g\in\omom$, put $\psi_-(g)=A_g=\{g(0),g(g(0)),\ldots ,g^k(0),\ldots\}$. If $g$ is not of this form,
choose some strictly increasing $g^\prime>g$ and
put $\psi_-(g)=\psi_-(g^\prime)$.
Then the implication 
$g\geq^* f_\pi\implies A_g P\pi$ holds, so
$(\psi_-,\psi_+)$ gives the desired Tukey reduction.

Now, it suffices to show that $( \omoms, \Sym (\omega), \centernot P )\leq_T( D_X^\rel , \Sym (\omega) , R_Y^\rel)$. To this end, first let $\varphi_+$ be the identity on $\Sym(\omega)$.  To 
define $\varphi_-:\omoms\to D_X^\rel$, fix, for each $B\in\omoms$, a set
$A^B\subseteq B$ such that $d_B(A^B)\in X$, and put $\varphi_-(B)=(A^B,B)$.
Suppose $B P\pi$. Then $\pi$ preserves $B$, so
$d_{\pi(B)}(\pi(A^B))=d_B(A^B)$, and 
$\varphi_-(B) R_{Y}^{\rel}\pi$ fails. This completes the proof of the theorem.
\end{proof}

\begin{theorem} \label{reldensity-dominating}
If $\osc \notin Y$ and $X \cap [0,1] \neq \emptyset$, then $(\omom, \omom, \leq^* ) \leq_T ( D_X^\rel , \Sym (\omega) , R_Y^\rel)$.
\end{theorem}

If $X = \{ \osc \}$ and $0 \notin Y$ or $1 \notin Y$, this also 
holds by Theorem~\ref{density-continuum}. We do not know 
whether this is still true if $X = \{ \osc \}$ and $\{ 0,1 \} \sub Y$.

\begin{corollary}
 If $\osc \notin Y$ and $X \cap [0,1] \neq \emptyset$, then $\dd\leq\ddd_{X,Y}^{\rel}$.
\end{corollary}

\begin{proof}[Proof of Theorem~\ref{reldensity-dominating}]
 The case $\osc\notin X$ clearly suffices, and an argument analogous to the one in ~\cite{BBBHHL20} for
 $\dd\leq\rrr_{fi}$ works. We provide the necessary adaptations to get a Tukey reduction.

 So, suppose $\osc\notin X$. For $g\in\omom$ and $\pi\in\Sym(\omega)$, let
 $\varphi_-(g)=(A_g, B_g)$ and $\varphi_+(\pi)=f_{\pi}$, where $B_g$ and $f_{\pi}$ 
 are as in the last proof, and $d_{B_g}(A_g)\in X$.
 Now, to get the implication 
 $(A_g, B_g)R_Y^\rel \pi\implies g\leq^* f_{\pi}$, note that if $g\not\leq^* f_{\pi}$
 then the sequence $\left( { |\pi[A_g] \cap n| \over |\pi[B_g] \cap n | }: n\in\omega \right)$ has a subsequence converging to
 $d_{B_g}(A_g)$. So $(A_g, B_g)R_Y^\rel \pi$ fails, and this completes the proof.
\end{proof}





The above results suggest the question of how far 
the analogy between $\rrr_{X,Y}$ and $\ddd_{X,Y}^{\rel}$
goes. In particular, one could ask:

\begin{question}   \label{ddversusrr}
Is $\rrr=\ddd_{(0,1),\all}^{\rel}$?
\end{question}

The point is that both cardinals have the same lower bounds 
$\bb$ and $\cov (\N)$ and the same upper bound $\non (\M)$
(see~\cite{BBBHHL20} and Theorem~\ref{density-covnull},
Proposition~\ref{reldensity-nonmeager}, and Theorem~\ref{reldensity-unbounding}). Analogously one may ask whether $\rrr_i = \ddd_{(0,1) ,
\{ 0,1 \} }^{\rel}$ or $\rrr_f = \ddd_{(0,1) , ( 0,1 ) }^{\rel}$.


We may also look at the similarity or nonsimilarity between
rearrangement numbers and relative density numbers  by
considering the set of permutations that leave all conditionally
convergent series (all relative densities, resp.) fixed.
The former have been studied in a number of papers (e.g.~\cite{Ag55}
or~\cite{GGRR88}), while a connection between the two has been
established by Garibay, Greenberg, Resendis, and Rivaud~\cite{GGRR88}. 
But, alas, they have a different notion of relative density!
Let $\{ b_i : i \in \omega \}$ be the increasing enumeration of $B$.
Say that $A \sub B$ has {\em strong relative density} $r$ in $B$,
$sd_{B} (A) = r$, if given any $\varepsilon > 0$ there is $N$
such that if $m - n > N$, then 
\[ \left| {  | A \cap \{ b_n, b_{n+1}, ... , b_{m-1} \} | 
\over m - n } - r \right| < \varepsilon\]
Note that $sd_B (A) = r$ implies $d_B (A) = r$ but not vice-versa.
Say that $\pi \in \Sym (\omega)$ {\em preserves c.c. series} 
({\em preserves (strong) density}, resp.) if given any
conditionally convergent series $\sum a_n$ (any sets $A \sub B$
such that the (strong) relative density $(s)d_B (A)$ exists,
resp.), $\sum_n a_{\pi (n)} = \sum a_n$ ($(s)d_{\pi [B] }
(\pi [A]) = (s)d_B (A)$, resp.) holds. We need some more
combinatorial notions:

\begin{definition}
    \begin{enumerate}
        \item For finite subsets $M,N \sub \omega$ write $M < N$
        if $\max (M) < \min (N)$.
        \item \cite[Definition 1.2]{GGRR88} Two sets $M = \{ m_0 < ... < m_k\}$
        and $N = \{ n_0 < ... < n_k \}$ of natural numbers of the
        same size are {\em collated} if $m_0 <
        n_0 < m_1 < ... < m_k < n_k$.
        \item \cite[Definition 1.3]{GGRR88} $\pi \in \Sym (\omega)$
        {\em satisfies condition} $\mathrm{A}$ if there exists $k \in \omega$ 
        such that whenever $M$ and $N$ are collated and $\pi[M] <
        \pi [N]$ then $|N| = |M| < k$.
        \item $\pi \in \Sym (\omega)$ {\em satisfies condition} $\mathrm{B}$ if
        there exists $k \in \omega$ such that for any $M , N $
        with $|M| = |N| $, $M < N$, and $\pi[N] < \pi[M]$, we have
        $|M| = |N| < k$.
    \end{enumerate}
\end{definition}

It is easy to see that the permutations satisfying condition $\mathrm{B}$ form
a subgroup of $\Sym (\omega)$. On the other hand, inverses of
permutations with condition $\mathrm{A}$ do not necessarily have condition 
$\mathrm{A}$~\cite[Example 1.8]{GGRR88}. Furthermore, if $\pi$ satisfies 
condition $\mathrm{B}$ it also satisfies condition $\mathrm{A}$, while it is easy to see
there are $\pi$ such that both $\pi$ and $\pi^{-1}$ satisfy condition
$\mathrm{A}$ but not condition $\mathrm{B}$. The main result of~\cite{GGRR88} is:

\begin{theorem}[Garibay, Greenberg, Resendis, and Rivaud] For $\pi \in \Sym (\omega)$, the following are
equivalent:
\begin{enumerate}
    \item $\pi$ satisfies condition $\mathrm{A}$,
    \item there exists $k \in \omega$ such that for every $n$,
    $\pi^{-1} [n]$ is a union of at most $k$ intervals,
    \item $\pi^{-1}$ preserves c.c. series,
    \item $\pi$ preserves strong density.
\end{enumerate}
\end{theorem}

The equivalence of (2) and (3) is originally due to Agnew~\cite{Ag55},
and the equivalence of the two combinatorial conditions (1) and
(2) is relatively easy to see (see also~\cite[Proposition 2.2]{GGRR88}).

\begin{theorem}
    For $\pi \in \Sym (\omega)$, the following are equivalent:
    \begin{enumerate}
        \item $\pi$ satisfies condition $\mathrm{B}$,
        \item $\pi$ preserves density.
    \end{enumerate}
\end{theorem}

\begin{proof}
(1) $\Longrightarrow$ (2): Assume $\pi$ does not preserve density.
So there are $A \sub B \sub \omega$ such that $r: = d_B (A)$
is defined and $d_{\pi [B]} ( \pi [A])$ is distinct (possibly
undefined). Without loss of generality $\overline{d}_{\pi [B]}
(\pi [A] ) > d_B (A) = r$. Let $s > r$ be such that for infinitely
many $k$, ${ | \pi[A] \cap k | \over | \pi [B] \cap k | } \geq s$.
Let $\varepsilon : = { s-r \over 2} $, and choose $\ell^*$ such
that for all $\ell \geq \ell^*$, ${ |A \cap \ell | \over |B \cap \ell 
|} < r + \varepsilon$. Let $(k_i : i \in \omega)$ be an
increasing enumeration of $k$ with ${ | \pi[A] \cap k | \over | \pi [B] \cap k | } \geq s$
and such that $| \pi [B] \cap k_0 | \geq | B \cap \ell^*|$.
Let $(\ell_i : i \in \omega )$ and $(m_i : i \in \omega)$
be such that $m_i = | B \cap \ell_i | = | \pi [B] \cap k_i |$.
Since 
\[ { | \pi[A] \cap k_i | \over m_i } \geq s > { r+s \over 2} >
{ | A \cap \ell_i | \over m_i }\]
there must be $A_i \sub A \setminus \ell_i$ of size $> \varepsilon
\cdot  m_i$ such that $\pi [A_i] \sub k_i$ and $B_i \sub B \cap 
\ell_i$ of the same size such that $\pi [B_i] \sub \omega \setminus
k_i$. Hence $B_i < A_i$ and $\pi [A_i] < \pi [B_i]$ and condition $\mathrm{B}$ fails.

(2) $\Longrightarrow$ (1): Recursively define $(k_i : i \in \omega)$
such that $k_0 = 2$ and $k_{i+1} = 2^{i+1} \sum_{j \leq i} k_j$.
Let $\ell_i = \sum_{j\leq i} k_j$; so $\ell_0 = 0$ and $k_{i+1} =
2^{i+1} \cdot \ell_{i+1}$. Assume $\pi$ does not satisfy 
condition $\mathrm{B}$. Then there are finite subsets $N_i \sub \omega$
and $M_i \sub \omega$ such that $|N_i| = |M_i| = k_i$,
$M_i < N_i < M_{i+1}$, and $\pi[N_i] < \pi [M_i] < \pi [ N_{i+1}]$,
$i \in \omega$.
Let $A_i \sub N_i$ of size $\ell_i$ such that $\pi [A_i] $ is
an initial segment of $\pi [N_i]$. Let $B_i \sub N_i \setminus
A_i$ be arbitrary of size $\ell_i$. Let $k_0' = k_0 = 2$ and
$k_i' = (2^i - 2) \ell_i = k_i - 2 \ell_i$ for $i > 0$.
Let $(n_{i,j} : j < k_i')$ be the increasing enumeration of
$N_i \setminus (A_i \cup B_i)$ and let $(m_{i,j} : j < k_i)$
be the increasing enumeration of $M_i$. Let $B = \bigcup_{i
\in \omega} (N_i \cup M_i)$ and
\[ A = \bigcup_{i \in \omega} \left( A_i \cup \left\{ n_{i,2j} :
j < {k_i' \over 2} \right\} \cup \left\{ m_{i,2j} : j < {k_i \over 2} \right\}
\right) \]

We first check that $d_B (A) = {1 \over 2}$. Indeed, let $m \in \omega$.
Let $i$ be minimal such that $M_i < m$. If also $N_i < m$, we easily 
see that
\[ {1 \over 2 } \leq {| A \cap m| \over | B \cap m| } \leq { |B \cap m
| + 2\over 2 | B \cap m | } \]
Otherwise let $\bar m_i = \min N_i$ and note that $|B \cap \bar m_i|
= 2 \ell_i + k_i = (2^i + 2) \ell_i$ and, similarly, $|A \cap \bar m_i|
= (2^{i-1} + 1) \ell_i$. Then we have:
\[ { (2^{i-1} +1) \ell_i + { |N_i \cap m| -\ell_i \over 2 }
\over (2^i + 2 ) \ell_i + |N_i \cap m| } \leq
{  |A\cap m| \over | B \cap m| } \leq
{ (2^{i-1} +1) \ell_i + { |N_i \cap m| +\ell_i +2 \over 2 }
\over (2^i + 2 ) \ell_i + |N_i \cap m| } \]
Clearly, the upper and lower bounds converge in both
cases to ${1 \over 2}$ as $i$ goes to $\infty$.

On the other hand, it is easy to see that, letting $\bar a_i = \max \pi 
[A_i] + 1$ and $\bar n_i = \min  \pi [M_i]$, 
\[ { | \pi [A] \cap \bar a_i | \over | \pi [B] \cap \bar a_i |} =
{2 \ell_i \over 3 \ell_i} = {2 \over 3} \]
and
\[ { |\pi [A] \cap \bar n_i | \over | \pi [B] \cap \bar n_i | } =
{ \ell_i + {k_i \over 2} \over 2\ell_i + k_i } = {1 \over 2} \]
so that $\overline{d}_{\pi[B]} (\pi [A]) \geq {2 \over 3}$ 
and $\underline{d}_{\pi[B]} (\pi [A]) \leq {1 \over 2}$. Thus,
$\pi$ does not preserve density.
\end{proof}

These results show that the set of permutations preserving c.c. series and density
are actually distinct. We have no idea whether this means that we can also
distinguish the two concepts, rearrangement of c.c. series and of relative
density, on the level of cardinal invariants 
(see Question~\ref{ddversusrr}).


Still regarding the analogy between $\rrr_{X,Y}$
and $\ddd^\rel_{X,Y}$, in both cases it is natural
to expect these cardinals to be big if $Y$ is 
small in some sense. In the case of the former, one
has:

\begin{theorem}\label{cof}
If $Y \subseteq \mathbb{R}$ and $|Y|<\mathfrak{c}$, then $\mathfrak{rr}_{\RR,Y}  = \mathfrak{c}$. 
\end{theorem}

\begin{proof}

Fix $Y \subseteq \mathbb{R}$ and let $\mathcal{F}$ be a family witnessing $\rrr_{\RR,Y}$. 
Also, let $\sum_n a_n$ be a c.c. series, say converging to $a$. 

For each $t \in \mathbb{R}$, let $\sum_n c_n^t$ be a series converging absolutely to $t-a$. Putting $a^t_n = a_n + c_n^t$, we get that $\sum_n a^t_n$ is a c.c. series converging to $t$. 
For each $t \in \mathbb{R}$, we can choose $\pi^t \in \mathcal{F}$ and $x^t \in Y$, $x^t \neq t$, such that $\sum_n a^t_{\pi^t (n)}$ converges to $x^t$. 

Now, $\mathbb{R} = \bigcup_{x \in Y} \{t \in \mathbb{R} : x^t=x\}$. Since $|Y|< \mathfrak{c}$, for every $\kappa < \mathfrak{c}$ there is $s \in Y$ such that the set $S_s :=\{t \in \mathbb{R}: x^t=s\}$ has cardinality strictly larger than $\kappa$. 

Let us show that the function $t\mapsto \pi^t$ is injective when restricted to any $S_s$,
which guarantees that $\mathcal{F}$ has cardinality $\mathfrak{c}$.

Indeed, let $t_1, t_2 \in S$ (i.e., $x^{t_i}=s$) and suppose $\pi^{t_1} = \pi^{t_2} =: \pi$. Then 
$$\left(\sum_n a_{\pi(n)} \right)+t_1-a=\sum_n a^{t_1}_{\pi(n)}=s=\sum_n a^{t_2}_{\pi(n)}= \left(\sum_n a_{\pi(n)} \right)+t_2-a$$
so $t_1=t_2$.
\end{proof}

It does not appear obvious how to
adapt the above line of reasoning to $\ddd_{(0,1),Y}^\rel$. Indeed, the fact that one can add an
absolutely convergent series to an arbitrary series without
essentially changing any information related to the
convergence of the latter, even after rearranging the terms,
is essential in the above proof. The set of 
infinite-coinfinite subsets of $\omega$ having asymptotic 
density lacks this structure. 

\begin{question}
Suppose $Y\subseteq(0,1)$
is such that $|Y|<\cc$. Is $\ddd^\rel_{(0,1),Y}=\cc$?
\end{question}

Still, one could reestablish the analogy by considering
the following: for a real sequence $(a_0, a_1,a_2,\dots)$, let
$$\mu(a_0 , a_1,a_2,\ldots)=\lim_{n\to\infty}\frac{a_0+\ldots+a_{n-1}}{n}$$
be its {\em asymptotic mean}, if the limit exists. The sequences
that have the same asymptotic mean regardless of the order
of the terms are the ones that converge. The following is
analogous to Riemann's Rearrangement Theorem:

\begin{theorem} \label{riemann-mean}

Let $r=(r_0,r_1,r_2,\ldots)$ be a sequence of real numbers having asymptotic mean.
Then,  there is $\pi\in S_\omega$ such that
 $\mu(r_{\pi(0)}, r_{\pi(1)},r_{\pi(2)},\ldots)=m$ if, and only if

$$\liminf(r)\leq m \leq\limsup(r).$$

\end{theorem}

We use the fact that if a real sequence has asymptotic mean and the terms of another 
real sequence are put in a sufficiently sparse set of indices, 
then the mean does not change:

\begin{lemma}  \label{sparsesets}

Let $(x_0,x_1,\ldots)$ and $(y_0,y_1,\ldots)$ be sequences of real numbers such that $\mu(x_0,x_1,\ldots)=m$. Then there is a set 
$A=\{a_0<a_1<\ldots\}\subseteq\omega$ such that, if 
$\omega\setminus A=\{b_0<b_1<\ldots\}$ and $(z_0,z_1,\ldots)$ is
defined by

\[ z_k = \left\{ \begin{array}{ll} {x_n} & \mbox{ if } k = b_n \\
    {y_n} & \mbox{ if } k = a_n \end{array} \right. \]
then $\mu(z_0,z_1,\ldots)=m$.
    
\end{lemma}

\begin{proof}

Let the $N$-th partial sum of $(z_0,z_1,\ldots)$ be
$$\sum_{k=0}^{N-1} z_k=\sum_{k=0}^{n-1} x_k+\sum_{k=0}^{\ell-1} y_k.$$
Note that
$$\frac{1}{N}\sum_{k=0}^{N-1} z_k-\frac{1}{n}\sum_{k=0}^{n-1} x_k=\frac{1}{N}\sum_{k=0}^{\ell-1} y_k-\frac{\ell}{N}\left(\frac{1}{n}\sum_{k=0}^{n-1} x_k\right).$$

Choosing the indices of $A$ amounts to choosing $\ell=\ell(N)$ for
each $N\in\omega$. The above equality shows that the conclusion of the Lemma will hold if $\ell(N)$ is 
such that 
$$\lim_{N\to\infty}\frac{1}{N}\sum_{k=0}^{\ell-1} y_k=\lim_{N\to\infty}\frac{\ell}{N}=0.$$
It is clearly possible to choose $\ell(N)$ to grow so slowly that
the above is true.
\end{proof}

\begin{proof}[Proof of Theorem~\ref{riemann-mean}]

The ``only if" part is clear.

Now, suppose $s=\liminf(r)$, $S=\limsup(r)$, and $s\leq m\leq S$.

If $m=s$, fix a subsequence $(x_0,x_1\ldots)$ of $r$
such that $\lim_{n\to\infty}x_n=s$ and let the remaining terms of $r$ be $(y_0,y_1\ldots)$. The sequence $(z_0,z_1\ldots)$ given
by the above Lemma is a rearrangement of $r$ with mean equal to $s$. The
case $m=S$ is treated similarly.

Now, suppose $s<m<S$. Let
$u=(u_0,u_1\ldots)$ and $v=(v_0,v_1,\ldots)$ be subsequences
of $r$ converging to $s$ and $S$ respectively. Without loss
of generality, suppose the indices of $u$ and $v$ are disjoint and the remaining terms of $r$ form the subsequence $z=(z_0,z_1\ldots)$.

First, we form a sequence $(x_0,x_1\ldots)$ using only terms from $u$ and $v$, in a way that is similar to 
the traditional proof of Riemann's Theorem:
Always respecting the order, keep adding terms from $v$ until 
the first time $\frac{x_0+\ldots+x_{n-1}}{n}>m$. This is possible,
since $v$ converges to $S>m$, Then keep adding terms from
$u$ until the first time $\frac{x_0+\ldots+x_{l-1}}{l}<m$, which
is possible, since $x$ converges to $s<m$. Proceeding in this way, clearly
$\lim_{n\to\infty}\frac{x_0+\ldots+x_{n-1}}{n}=m$.

Now, put the remaining terms $z_1,z_2\ldots$ in the
sparse set of indices $A$ given by the Lemma to get a permuted sequence $c=(c_0,c_1\ldots)$ such that $\mu(c)=m$.
\end{proof}

\begin{definition}
Let $\mathcal{D}$ be the set of sequences of real numbers having an asymptotic
mean. $\mathfrak{mm}$
is the minimal cardinality of a family $\mathcal{F}\subseteq \Sym(\omega)$
such that, for every $a\in\mathcal{D}$, there is $\pi\in\mathcal{F}$ such that $\mu(\pi[a])\neq\mu(a)$.
Analogously, one can define $\mathfrak{mm}_{X,Y}$ as above.
\end{definition}

The above definition clearly implies $\mathfrak{mm}\geq\ddd$. An argument similar to the first (easier) part of the proof of Theorem ~\ref{density-nonmeager} gives $\mathfrak{mm}\leq\non(\mathcal{M})$, so 
 $\mathfrak{mm}=\ddd=\non(\mathcal{M})$. Still, one can define the variants $\mathfrak{mm}_{X,Y}$, using the obvious definitions, and ask how
these cardinal behave in comparison to $\ddd_{X,Y}$. For instance, in direct analogy to
Theorem \ref{cof}, one has:

\begin{theorem}
If $Y \subseteq \mathbb{R}$ and $|Y|<\mathfrak{c}$, then $\mathfrak{mm}_{\mathbb{R},Y}  = \mathfrak{c}$. 
\end{theorem}

\begin{proof}

The argument is analogous to the one in the proof of Theorem~\ref{cof}. We point out the necessary adaptations.

Fix $Y \subseteq \mathbb{R}$ and let $\mathcal{F}$ be a family witnessing $\mathfrak{mm}_{\RR,Y}$. 
Also, let $(a_0,a_1\ldots)$ be a sequence of real numbers, with asymptotic mean equal to $a$. 

For each $t \in \mathbb{R}$, let $(c_0^t,c_1^t\ldots)$ be sequence of real numbers converging to $t-a$. Putting $a^t_n = a_n + c_n^t$, we get that $(a_0^t,a_1^t\ldots)$ has asymptotic mean equal to $t$. 
For each $t \in \mathbb{R}$, we can choose $\pi^t \in \mathcal{F}$ and $x^t \in Y$, $x^t \neq t$, such that $(a^t_{\pi^t (0)},a^t_{\pi^t (1)}\ldots)$ has mean equal to $x^t$. 

Now, the proof proceeds exactly like in the proof of Theorem~\ref{cof}.
\end{proof}







\end{document}